\documentclass{amsart}

\usepackage{amssymb}

\usepackage{epsfig}         

\newtheorem{thm}{Theorem}[section]
\newtheorem{prop}[thm]{Proposition}

\newtheorem{cor}[thm]{Corollary}

\theoremstyle{definition}
\newtheorem{definition}[thm]{Definition}

\theoremstyle{remark}

\numberwithin{equation}{section}
\newcommand{\N}{\mathbb{N}}  
\newcommand{\R}{\mathbb{R}}  
\newcommand{\C}{\mathbb{C}}  

\begin{document}



\title{Autonomous linear neutral equations with bounded Borel functions as initial data}

\author{Hans-Otto Walther}

\address{Mathematisches Institut, Universit\"{a}t Gie{\ss}en,
Arndtstr. 2, D 35392 Gie{\ss}en, Germany. E-mail {\tt
Hans-Otto.Walther@math.uni-giessen.de}}

\begin{abstract}

For equations with the neutral term strictly delayed we construct the fundamental solution, derive a variation-of-constants formula for inhomogeneous equations, and prove growth estimates. Only unavoidable measure and integration theory, up to the Riesz Representation Theorem, is used. A key notion is pointwise convergence of bounded sequences of continuous functions.

\bigskip


\bigskip


\end{abstract}

\maketitle

\section{Introduction}


These notes do not contain new results but work out an approach to the fundamental solution and to the variation-of-constants formula for neutral equations which is based on earlier work \cite{W0,W01} for simpler retarded functional differential equations. A brief sketch of the present account is contained in the appendix in the paper \cite{W9}, which provides a  principle of linearized stability for a class of neutral equations with state-dependent delay.

In the sequel we use only the unavoidable parts of measure and integration theory, up to the Riesz Representation Theorem. The source for this is Rudin's textbook \cite{Ru}.

Let $n\in\N$ and $h>0$ be given. Let $I=[-h,0]$. Consider linear autonomous neutral delay differential equations of the form
\begin{equation}
v'(t)=L\partial v_t+Rv_t
\end{equation}
and
\begin{equation}
\frac{d}{dt}(v-L\circ V)(t)=Rv_t,\quad V(t)=v_t,
\end{equation}
with continuous linear maps $L$ and $R$ from the complex Banach space
$C_{cn}$ of continuous maps $I\to\C^n$ into $\C^n$. The norm on $C_{cn}$ is given by
$|\phi|=\max_{t\in I}|\phi(t)|$. We also need the the Banach space $C^1_{cn}$ of
continuously differentiable maps $I\to\C^n$, with the norm given by
$|\phi|_1=|\phi|+|\partial\phi|$ and $\partial:C^1_{cn}\to C_{cn}$ being differentiation.
Differentiation is also indicated by a prime.
Segments $v_t:I\to N$ of a map $v:M\to N$ are defined by $v_t(s)=v(t+s)$, in case
$[t-h,t]\subset M$ for $t\in\R$.

A solution of Eq. (1.1) is defined to be a continuously differentiable map $v:[-h,\infty)\to\C^n$ so that Eq. (1.1) holds for all $t\ge0$. In particular, all segments
$v_t$, $t\ge0$, of a solution belong to the closed subspace
$$
C^1_{cnLR}=\{\phi\in C^1_{cn}:\phi'(0)=L\partial\phi+R\phi\}
$$
of $C^1_{cn}$. A solution of Eq. (1.2) is a continuous map $v:[-h,\infty)\to\C^n$ for which the map $[0,\infty)\ni t\mapsto v(t)-Lv_t\in\C^n$ is differentiable and satisfies Eq. (1.2) for all $t\ge0$.

Equation (1.1) arises as linear variational equation at a stationary point of a semiflow which is given by a neutral equation of the form
$$
x'(t)=g(\partial\,x_t,x_t),
$$
under mild hypotheses on $g$ designed to cover examples with state-dependent delay \cite{W8}.

Linear neutral equations in the form (1.2) are familiar from the work of Hale and Meyer \cite{HM} and Henry \cite{He}.
That the form (1.2) of linear neutral equations is useful can be seen already in Poisson's paper \cite{P}.

\begin{prop}
Every solution of Eq. (1.1) also is a solution of Eq. (1.2).
\end{prop}

\begin{proof} Suppose $v$ is a solution of
Eq. (1.1). Then the curve $V:[0,\infty)\ni t\mapsto v_t\in
C_{cn}$ is differentiable with $DV(t)1=\partial v_t$ for all
$t\ge0$. As $L$ is linear and continuous we infer that the map
$L\circ V:[0,\infty)\ni t\mapsto Lv_t\in\C^n$ is differentiable,
with $D(L\circ V)(t)1=L\partial v_t$ for all $t\ge0$. Hence
$v|[0,\infty)-L\circ V$ is differentiable, and Eq. (1.1) shows that
for every $t\ge0$ Eq. (1.2) holds.
\end{proof}

A hypothesis which we make throughout the paper is that the neutral part in Equations (1.1) and (1.2) is {\it strictly delayed}, or, that $L\phi$ does not depend on the values of $\phi$ for $t\ge0$ small. More precisely, we require that

($\Delta$) there is $\Delta\in(0,h)$ with $L\phi=0$ for all $\phi\in C_{cn}$ with $\phi(t)=0$
on $[-h,-\Delta]$.

Corollary 2.8 from \cite{W8} shows that the property ($\Delta$) holds for maps $L$ in Eq. (1.1) as they arise in linearization of a semiflow which is given by the neutral equation $x'(t)=g(\partial\,x_t,x_t)$. This semiflow was constructed in \cite{W8} under a hypothesis similar to ($\Delta$), for nonlinear maps and local in nature. Related hypotheses have been used in the study of neutral equations for a long time, see the monographs \cite{H,HVL} and also \cite{NH}. Property ($\Delta$) is a rather strong condition among them.


We now abandon Eq. (1.1) and consider Eq. (1.2), and its inhomogeneous version
\begin{equation}
\frac{d}{dt}(v-L\circ V)(t)=Rv_t+q(t),\quad V(t)=v_t,
\end{equation}
with a continuous map $q:[0,\infty)\to\C^n$. Initial values $\phi\in C_{cn}$ uniquely define solutions of Eq. (1.2), see Proposition 5.1 (whose proof would also yield solutions to the inhomogeneous equation, but we don't need this). We are interested in more general solutions with discontinuous initial data, in particular for the data $\phi:I\to\C^n$ given by $\phi(t)=0$ for $t<0$ and $\phi(0)=e_j$, with $e_{jk}=\delta_{jk}$ for $j,k$ in $\{1,\ldots,n\}$.


A notion which is important in the sequel is pointwise convergence of a uniformly bounded sequence of continuous functions on an interval $J\subset\R$. Let us underline that this means convergence everywhere, as opposed to {\it convergence almost everywhere} with respect to a $\sigma$-algebra, which would be less useful for our purpose. We say a map $w:J\to X$, $X$ a normed linear space over the field $\C$, has property (A) if it is the pointwise limit of a
sequence of continuous maps $w_m:J\to X$, $m\in\N$, which is uniformly bounded, $\sup_{m,t}|w_m(t)|<\infty$. In case $X=\C$ each function with property (A) is a bounded Borel function, i. e., bounded and measurable with respect to the $\sigma$-Algebra of Borel sets in $J$. In other words, for $X=\C$ we deal with functions in the Baire classes 0 and 1 \cite{N}.

The organization of this paper is as follows. Section 2 borrows from \cite{Ru} the parts of measure theory
used in the sequel. Section 3 defines extensions of operators $C_{cn}\to Y$  with range in a normed linear space $Y$ over $\C$ to the larger domain $B_{nA}$ of maps whose components have property (A). The extensions have range in the second dual $Y^{\ast\ast}$. - For the purpose of the present account it would be sufficient to consider operators $C_{cn}\to\C^n$, or operators with finite-dimensional range, but we prefer to work in a slightly more general
framework.

For solutions of inhomogeneous equations we need integrals of curves with range in the second dual of a normed linear space. The corresponding weak-star integral is introduced in Section 4. The main theme of this section is commutativity of the integral with operators and their extensions.

Section 5 establishes existence, uniqueness and continuous dependence for solutions to Eq. (1.2) with initial data in the space $B_{nA}$. As a special case we obtain the fundamental solution and an estimate of its growth.

Section 6 deals with the inhomogeneous equation (1.3).
The variation-of-constants formula for solutions with initial data in $C_{cn}$ is derived (Corollary 6.3) and yields a growth estimate (Corollary 6.4).



{\bf Notation.} $\delta_{jk}$ is the Kronecker symbol with values
$1$ for $j=k$ and $0$ for $j\neq k$. For a set $X$ the set of its
subsets is denoted by ${\mathcal P}(X)$. The set of maps with
domain $D$ and range $R$ is denoted by $R^D$. We write
$[-\infty,\infty]$ for $\R\cup\{-\infty,\infty\}$ and $[0,\infty]$
for $[0,\infty)\cup\{\infty\}$. $S^1_{\C}$ stands for the unit
circle in $\C$. The projection $\C^n\to\C$ of a row vector onto
its $j$-th component is denoted by $pr_j$. The $j$-th component of
a vector in $z\in\C^n$ is indicated by the index $j$. The same
notation is used in case of vector-valued maps.

For normed linear spaces $X,Y$ over the same field ($\R$ or $\C$) the
normed linear space of continuous linear maps $X\to Y$, with the norm given by
$$
|T|=\sup_{|x|=1}|Tx|,
$$
is denoted by
$L_c(X,Y)$. The dual of $X$ is $X^{\ast}=L_c(X,\C)$. The linear map
$$
\iota:X\to X^{\ast\ast}
$$
given by $\iota(x)x^{\ast}=x^{\ast}x$ preserves the norm.

It is convenient to set $C_c=C_{c1}$.

We shall use the Riemann integral of a continuous map from a
compact interval into a Banach space and the fact that this
integral commutes with continuous linear maps.
\section{Review of measure theory}

We give a brief account of what we need in the sequel, following
Rudin's monograph \cite{Ru}. See also \cite{A}, and compare Appendix I in \cite{DvGVLW}.

{\bf $\sigma$-algebras, measurable functions, measures.} A
$\sigma$-Algebra $S$ on a set $X$ is a collection $S\subset{\mathcal
P}(X)$ of subsets of $X$ so that $X\in S$, $X\setminus A\in S$ for
all $A\in S$, and for all sequences $(A_n)_1^{\infty}$ in $S$,
$\cup_{n=1}^{\infty}A_n\in S$. Then the pair $(X,S)$ is called a
measurable space.

Let $(X,S)$ be a measurable space and let $Y$ be a topological
space. A map $f:X\to Y$ is measurable iff preimages of open sets are
measurable. If $f:X\to\C$ is measurable then also the real functions
$Re\,f$, $Im\,f$, $|f|$ are measurable. For measurable functions
$f:X\to\R$ and $g:X\to\R$ also $f+ig:X\to\C$ is measurable. For
$f:X\to\C$ and $g:X\to\C$ measurable also $f+g$, $fg$ and
$f+ig:X\to\C$ are measurable. For $f:X\to\C$ measurable there is a
measurable function $\alpha:X\to\C$ with $\alpha(X)\subset S^1_{\C}$
and
$$
f(x)=\alpha(x)|f(x)|\quad\text{on}\quad X.
$$
For $A\in S$ the characteristic function $\chi_A:X\to\R$ given by
$\chi_A(x)=1$ on $A$ and $\chi_A(x)=0$ elsewhere is measurable.

Let $X$ be a set. For every set $F\subset {\mathcal P}(X)$ there is
a smallest $\sigma$-Algebra containing $F$.

If $X$ is a topological space then the smallest $\sigma$-Algebra
containing the topology is called the Borel-$\sigma$-Algebra of $X$.
It contains all open and all closed sets. The elements of the
Borel-$\sigma$-algebra are called Borel sets. For maps $X\to Y$ into
a topological space $Y$ we speak of Borel measurability.  All
continuous maps with domain $X$ are Borel-measurable.
Borel-measurable real- and complex-valued functions are called Borel
functions.

Let $(X,S)$ be a measurable space. Limits of pointwise convergent
sequences of measurable functions $X\to\C$ are measurable.

If $f:X\to\R\cup\{-\infty,\infty\}$ and
$g:X\to\R\cup\{-\infty,\infty\}$ are measurable then also
$\max\{f,g\}$ and $\min\{f,g\}$ are measurable.

A (positive) measure $\mu:S\to[0,\infty]$ on a measure space $(X,S)$
is a map which is countably additive, i. e.,
$$
\mu(\cup_{n=1}^{\infty}A_n)=\sum_{n=1}^{\infty}\mu(A_n)
$$
for any sequence of pairwise disjoint elements $A_n\in S$, and which satisfies
$\mu(A)<\infty$ for some $A\in S$. A complex measure $\mu:S\to\C$ is
countably additive. Real measures are complex measures with values
in $\R$.

A triple $(X,S,\mu)$ with a measurable space $(X,S)$ and a
(positive) measure $\mu:S\to[0,\infty]$ is called a measure space.

Let $X$ be a locally compact Hausdorff space. A (positive or complex
or real) measure on the Borel-$\sigma$-Algebra of $X$ is called a
Borel measure on $X$. Let $\mu$ be a positive Borel measure on $X$.
A Borel set $B\subset X$ is called

outer regular iff $\mu(B)=\inf\{\mu(V):B\subset V,
V\,\,\text{open}\}$,

and

inner regular iff $\mu(B)=\sup\{\mu(K):K\subset B,
K\,\,\text{compact}\}$.

The positive Borel measure $\mu$ is called regular if every Borel
set $B\subset X$ is outer regular and inner regular.

{\bf Integrals.} We use the following addition and  multiplication
rules:
\begin{eqnarray*}
r+\infty & = &
\infty+r=\infty,\,\,r-\infty=-\infty+r=-\infty\,\,\text{for
all}\,\,r\in\R,\\
r\cdot\infty & = & \infty\cdot r=\infty\,\,\text{for}\,\,0<r\le\infty,\\
0\cdot\infty & = & \infty\cdot 0=0.
\end{eqnarray*}
Let $(X,S,\mu)$ be a measure space. Simple functions
$s:X\to[0,\infty)$ are those which have only finitely many values.
Then
$$
s=\sum_{c\in s(X)}c\,\chi_{s^{-1}(c)}
$$
For a measurable simple function $s$ on $X$ and for $A\in S$, the
integral of $s$ with respect to $\mu$ is defined as
$$
\int_A sd\mu=\sum_{c\in s(X)}c\,\mu(s^{-1}(c)\cap A)
$$
For $f:X\to[0,\infty]$ measurable and for $A\in S$ the integral of
$f$ over $A$ is defined as
$$
\int_A fd\mu=\int_Af(x)d\mu(x)=\sup_{s\,\,\text{simple and measurable},\,\,0\le s\le
f}\int_A sd\mu.
$$
The set of measurable $f:X\to\C$ with $\int_X|f|d\mu<\infty$ is
denoted by $L^1(\mu)$. For each $f\in L^1(\mu)$ and for each $A\in
S$ the integrals of the positive and negative parts of $Re\,f$ and
$Im\,f$, considered as maps into $[0,\infty]$, belong to $\R$, and
we define
$$
\int_Af\,d\mu=\int_A(Re\,f)^+d\mu-\int_A(Re\,f)^-d\mu+i(\int_A(Im\,f)^+d\mu-\int_A(Im\,f)^-d\mu)\in\C
$$
The set $L^1(\mu)$ is a vector space, and the maps
$$
L^1(\mu)\ni f\mapsto\int_Afd\mu\in\C, A\in S,
$$
are linear. For all $f\in L^1(\mu)$,
$$
\left|\int_Xfd\mu\right|\le\int_X|f|d\mu.
$$

For measurable maps $f:X\to[-\infty,\infty]$ and for $A\in S$ the
integral of $f$ over $A$ is defined as
$$
\int_Af\,d\mu=\int_Af^+d\mu-\int_Af^-d\mu
$$
provided at least one of the integrals on the right hand side belongs to
$\R$ (is finite).

\begin{thm}
(Lebesgue's Dominated Convergence Theorem) Let a measure space
$(X,S,\mu)$ be given. Suppose the sequence of measurable functions
$f_n:X\to\C$, $n\in\N$, is (everywhere on $X$) pointwise convergent
to $f:X\to\C$, and for some real-valued $g\in L^1(\mu)$,
$$
|f_n(x)|\le g(x)\quad\text{for all}\quad n\in\N,x\in X.
$$
Then $f\in L^1(\mu)$ and for $n\to\infty$,
$$
\int_X|f_n-f|d\mu\to0,\quad\int_Xf_nd\mu\to\int_Xf\,d\mu.
$$
\end{thm}

{\bf The Riesz Representation Theorem.} The total variation of a
complex measure $\mu:S\to\C$, $(X,S)$ a measurable space, is given
by
$$
|\mu|(A)=\sup\left\{\sum_{j=1}^{\infty}|\mu(A_j)|:A=\cup_{j=1}^{\infty}A_j,
A_j\cap A_k=\emptyset\,\,\text{for}\,\,j\neq k,A_j\in S\right\}
$$
$|\mu|$ is a positive measure on $S$ with range in $[0,\infty)$.

Notice that each measurable function $h:X\to\C$ with values on the
unit circle belongs to $L^1(|\mu|)$ since
$$
\int_X|h|d|\mu|=\int_X\chi_Xd|\mu|=|\mu|(X)<\infty.
$$
Given a measure $\mu:S\to\C$, $(X,S)$ a measurable space, there
exists a measurable function $h:X\to\C$ with values on the unit
circle such that
$$
d\mu=h\,d|\mu|,
$$
which means
$$
\mu(A)=\int_Ah\,d|\mu|\quad\text{for all}\quad A\in S.
$$
For a complex Borel measure $\mu:S\to\C$, $X$ a locally compact
Hausdorff space, and $h:X\to\C$ as before, we define
$$
\int f\,d\mu=\int fh\,d|\mu|
$$
for all $f:X\to\C$ such that $fh\in L^1(|\mu|)$.

Let $X$ be a locally compact Hausdorff space. $C_c(X)$ is the vector
space of all continuous functions $X\to\C$ with compact support. Let $C_0(X)$ denote its closure in the Banach space of continuous bounded functions $X\to\C$, with the norm given by
$$
|f|=\sup_{x\in X}|f(x)|.
$$
All functions $f\in C_0(X)$ are bounded Borel functions, and therefore
$fh\in L^1(|\mu|)$ for every complex Borel measure (defined on the
Borel-$\sigma$-algebra of $X$) and every bounded Borel function
$h:X\to\C$.


\begin{thm}
{\it (Riesz Representation Theorem, Theorem 6.19 in \cite{Ru})} Let
$X$ be a locally compact Hausdorff space, consider $L:C_0(X)\to\C$
linear and continuous. Then there exist a uniquely determined
complex Borel measure $\mu$ with $|\mu|$ regular and a Borel
function $h:X\to\C$ which has all values on the unit circle such
that for every $f\in C_c(X)$,
$$
Lf=\int_Xf\,d\mu=\int_Xfh\,\,d|\mu|.
$$
\end{thm}

The map $L\mapsto\mu$ given by Theorem 2.2 is an isomorphism onto the $\C$-vectorspace of
regular complex Borel measures on $X$, compare Theorem IV.6.3 in \cite{DS}.

{\bf Lebesgue measure.} A cell $Q=Q(I_1,\ldots,I_n)$ in $\R^n$ is a
set given by the Cartesian product of $n$ bounded intervals
$I_1,\ldots,I_n$. Its volume $Vol(Q)$ is defined as the product of
the diameters of the intervals $I_j$.

A measure $\mu:S\to[0,\infty]$ is called {\it complete} if every
subset of a set of measure zero belongs to S.

There exist a $\sigma$-Algebra $M$ on $\R^n$ and a positive
complete measure $m:M\to[0,\infty]$ with the following properties.

(i) M contains all Borel sets, and $E\in M$ if and only if there
are a $F_{\sigma}$-set $A$ (countable union of closed sets) and a
$G_{\delta}$-set $B$ (countable intersection of open sets) so that
$A\subset E\subset B$ and $m(B\setminus A)=0$. The measure $m$ is
regular.

(ii) For every cell $Q\subset\R^n$,
$$
m(Q)=Vol(Q).
$$

(iii) For every $x\in\R^n$ and all $E\in M$, $x+E\in M$ and
$m(x+E)=m(E)$.

(iv) For any positive Borel measure $\mu$ with property (iii) (for
all Borel sets $E$) and such that $m(K)<\infty$ for all compact
sets $K\subset\R^n$ there exists $c\ge0$ such that
$\mu(E)=c\,m(E)$ for all Borel sets.

This measure $m$ is called the Lebesgue measure on $\R^n$. The
elements of $M$ are the Lebesgue measurable subsets of $\R^n$. If
$E\in M$ then the measure $m_E$ obtained by restricting $m$ to the
$\sigma$-algebra $M_E$ of the sets $A\cap E$, $A\in M$, is called
the Lebesgue measure on $E$. We use the notation
$$
L^1(E)=L^1(m_E)\,\,\text{and}\,\,\int_Efdm=\int_Ef(x)dm(x)=\int fdm_E.
$$

Lebesgue-integrals of maps with range in $\C^n$ are defined as the
vectors whose components are the integrals of the components of
the map. For such maps and for any norm on $\C^n$ we have the
estimate
$$
\left|\int_Efdm\right|\le\int_E|f|dm.
$$

Riemann-integrable functions $f:[a,b]\to\C$ belong to
$L^1([a,b])$, and
$$
\int_a^bf(x)dx=\int_{[a,b]}fdm=\int_{(a,b)}fdm
$$




A {\it regulated function} $f:[a,b]\to\C$, $a<b$, has right and left
limits at every $t\in[a,b]$, see \cite{D}. Ruled functions can be
approximated by sequences of step functions uniformly on $[a,b]$. A
step function is a simple function which is constant on each of the
open intervals $(t_{j-1},t_j)$, $j=1,\ldots,J$, given by a partition
$$
a=t_0<t_1<\ldots<t_J=b
$$
of the interval $[a,b]$.

{\bf Product measures and Fubini's Theorem.} Let a measurable space
$(X,S)$ be given. A (positive) measure $\mu$ on $X$ is
$\sigma$-finite if $X$ is the (countable) union of sets $A_n\in S$,
$n\in\N$, with $\mu(A_n)<\infty$. In this case $(X,S,\mu)$ is called
$\sigma$-finite.

Suppose $(X,S,\mu)$ and $(Y,T,\lambda)$ are $\sigma$-finite measure
spaces. The product $S\times T$ of the $\sigma$-algebras $S$ and $T$
is the smallest $\sigma$-algebra on $X\times Y$ containing all sets
$A\times B$ with $A\in S$ and $B\in T$.

For a function $f:X\times Y\to M$, $M$ some set, and for every $x\in
X$ and all $y\in Y$ define $f_x:Y\to M$ and $f^y:X\to M$ by
$$
f_x(\eta)=f(x,\eta)\,\,\text{and}\,\,f^y(\xi)=f(\xi,y).
$$
If $f:X\times Y\to\C$ (or, $f:X\times Y\to[0,\infty]$) is $(S\times
T)$-measurable then all functions $f_x$ are $T$-measurable and all
functions $f^y$ are $S$-measurable.

Let $Q\in S\times T$. Then all
$$
Q_x=\{y\in Y:(x,y)\in Q\}, x\in X,
$$
belong to $T$, and all
$$
Q^y=\{x\in X:(x,y)\in Q\}, y\in Y,
$$
belong to $S$; the function
$$
\phi:X\ni x\mapsto \lambda(Q_x)\in[0,\infty]
$$
is $S$-measurable; the function
$$
\psi:Y\ni y\mapsto \mu(Q^y)\in[0,\infty]
$$
is $T$-measurable; and
$$
\int_X\phi\,d\mu=\int_Y\psi\,d\lambda.
$$

The relations
$$
(\mu\times\lambda)(Q)=\int_X\lambda(Q_x)d\mu(x)=\int_Y\mu(Q^y)d\lambda(y)
$$
for $Q\in S\times T$ define a $\sigma$-finite measure
$\mu\times\lambda$, which is called the product measure of $\mu$
and $\lambda$.

\begin{thm}
(Fubini's Theorem) Suppose $(X,S,\mu)$ and $(Y,T,\lambda)$ are
$\sigma$-finite measure spaces, and $f:X\times Y\to\C$ ($f:X\times
Y\to[-\infty,\infty]$) is $(S\times T)$-measurable.

(a) In case $f(X\times Y)\subset[0,\infty]$,
$$
\phi:X\ni
x\mapsto\int_Yf_xd\lambda\in[0,\infty]\quad\text{is}\quad
S-\text{measurable},
$$
$$
\psi:Y\ni y\mapsto\int_Xf^yd\mu\in[0,\infty]\quad\text{is}\quad
T-\text{measurable},
$$
and
\begin{equation}
\int_X\phi\,d\mu=\int_{X\times Y}f\,d(\mu\times\lambda)=\int_Y\psi
d\lambda.
\end{equation}

(b) In case $f:X\times Y\to\C$ and
$$
\int_X\phi^{\ast}d\mu<\infty\quad\text{for}\quad\phi^{\ast}:X\ni
x\mapsto\int_Y|f_x|d\lambda\in[0,\infty]
$$
we have
$$
f\in L^1(\mu\times\lambda).
$$

(c) In case $f\in L^1(\mu\times\lambda)$ there exist sets $Z\in S$ and $N\in T$ of measure zero with
$f_x\in L^1(\lambda)$ for all $x\in X\setminus Z$ and $f^y\in L^1(\mu)$ for all $y\in Y\setminus N$, the maps
$\phi:X\to\C$ and $\psi:Y\to\C$ given by
\begin{eqnarray*}
\phi(x)=\int_Yf_xd\lambda & \text{for all} & x\in X\setminus Z\quad\text{and}\quad \phi(x)=0\quad\text{on}\quad Z,\\
\psi(y)=\int_Xf^yd\mu & \text{for all} & y\in Y\setminus N\quad\text{and}\quad \psi(y)=0\quad\text{on}\quad N,
\end{eqnarray*}
belong to $L^1(\mu)$ and $L^1(\lambda)$, respectively, and Eq. (2.1) holds.
\end{thm}

\begin{cor} {\it (Without product measure)} If $f:X\times Y\to\C$
(or, $f:X\times Y\to[-\infty,\infty]$)
is $(S\times T)$-measurable and
$$
\int_Xd\mu(x)\int_Y|f(x,y)|d\lambda(y)<\infty
$$
then $f\in L^1(\mu\times\lambda)$, and the assertions of part (c) of Theorem 2.3 hold.
\end{cor}

\section{Bounded Borel functions on the initial interval}


Let $B\supset C_c$ denote the normed linear space of bounded Borel functions
$I\to\C$. Let $\mu$ be a complex regular Borel measure on $I$, and
let $H\in B$ be given with $|H(t)|=1$ on the unit circle and
$$
\int fd\mu=\int fHd|\mu|\,\,\text{for all}\,\,f\in L^1(\mu).
$$

\begin{prop}
(i) $B\subset L^1(|\mu|)$.

(ii) Suppose the sequence $(\phi_j)_1^{\infty}$ in $B$ converges
pointwise to $\phi:I\to\C$, and $\sup_{j,t}|\phi_j(t)|<\infty$. Then
$\phi\in B$ and
$$
\int\phi_jd\mu\to\int\phi d\mu\,\,\text{as}\,\,j\to\infty.
$$
\end{prop}

\begin{proof}
1. On (i). For every $c\ge0$ the constant function $c\chi_I$ is
continuous, hence Borel-measurable, so belongs to $B$. As
$|\mu|(I)<\infty$,
$$
\int c\chi_Id|\mu|=c|\mu|(I)<\infty.
$$
For $\phi\in B$ we get $|\phi|\in B$ and (as $|\mu|$ is a Borel
measure)
$$
\int|\phi|d|\mu|\le\sup_t|\phi(t)|\int\chi_Id|\mu|<\infty,
$$
hence $\phi\in L^1(|\mu|)$.

2. On (ii). The sequence $(\phi_jH)_1^{\infty}$ in $B$ converges
pointwise to $\phi H$, and for all $j\in\N$ and $t\in I$,
$|(\phi_jH)(t)|\le\sup_{j,t}|\phi_j(t)|<\infty$; constant functions
belong to $B\subset L^1(|\mu|)$ (see (i)). Theorem 2.1 (the
Lebesgue Dominated Convergence Theorem) yields $\phi H\in
L^1(|\mu|)$ and $\int\phi_jd\mu=\int\phi_jHd|\mu|\to\int\phi
Hd|\mu|=\int\phi d\mu$ as $j\to\infty$.
\end{proof}

Let ${\mathcal L}\in C_c^{\ast}$ be given and let $\mu$ be the
complex Borel measure associated to ${\mathcal L}$ by Theorem 2.2 (the Riesz
Representation Theorem). We define a linear {\it extension} of ${\mathcal L}$
to a map on $B$ by
$$
{\mathcal L}_e\phi=\int\phi d\mu\quad (=\int\phi Hd|\mu|,\,\,\text{with}\,\,H\in
B\,\,\text{as above}).
$$

Proposition 3.1 (ii) yields the following result.

\begin{cor}
Let ${\mathcal L}\in C_c^{\ast}$ be given.
Suppose the sequence $(\phi_j)_1^{\infty}$ in $B$ converges
pointwise to $\phi:I\to\C$, and $\sup_{j,t}|\phi_j(t)|<\infty$. Then
$\phi\in B$ and
$$
{\mathcal L}_e\phi_j\to {\mathcal L}_e\phi\,\text{as}\,\,j\to\infty.
$$
\end{cor}

We proceed to vector-valued functions on $I$. Let $n\in\N$ and recall
$C_{cn}=C(I,\C^n)$ (with the maximum-norm), and define
$$
B_n=\{\phi\in I\times\C^n:\phi\,\,\text{is a map with}\,\,\phi_k\in
B\,\,\text{for}\,\,k=1,\ldots,n\}.
$$
$B_n\supset C_{cn}$ is a normed linear space over the field $\C$. It will be convenient
to equip $B_n$ with the supremum-norm given by
$$
|\phi|=\sup_t|\phi(t)|.
$$

Let $\lambda\in C_{cn}^{\ast}$ be given. We define a linear
{\it extension} $\lambda_e:B_n\to\C$ as follows. For
$k\in\{1,\ldots,n\}$ let $E_k:C_c\to C_{cn}$ denote the linear
continuous map (embedding) given by $(E_k\chi)_j=\delta_{jk}\chi$,
for $j=1,\ldots,n$.
For $k=1,\ldots,n$ define
$$
\lambda_k=\lambda\circ E_k,
$$
let $\mu_k$ denote the complex
Borel measure associated to $\lambda_k$ by Theorem 2.2, and let
$\lambda_{ke}:B\to\C$ denote the extension from above,
$$
\lambda_{ke}\chi=\int\chi d\mu_k\,\,\text{for}\,\,\chi\in B.
$$
Define $\lambda_e:B_n\to\C$ by
$$
\lambda_e\phi=\sum_{k=1}^n\lambda_{ke}\phi_k.
$$
Then $\lambda_e$ is linear, and for $\phi\in C_{cn}$ we have
\begin{eqnarray*}
\lambda_e\phi & = & \sum_{k=1}^n\lambda_{ke}\phi_k=\sum_{k=1}^n(\lambda\circ\,E_k)_e\phi_k=
\int\phi_kd\mu_k\\
& = & \sum_{k=1}^n(\lambda\circ\,E_k)\phi_k=\lambda\left(\sum_{k=1}^nE_k\phi_k\right)=\lambda\phi.
\end{eqnarray*}

In case of maps $\lambda_j\in C_{cn}^{\ast}$ and coefficients
$a_j\in \C$ for $j=1,\ldots,d$ we have
$$
\left(\sum_{j=1}^da_j\lambda_j\right)_e=\sum_{j=1}^da_j\lambda_{je}.
$$
We do not need this in the sequel and omit the proof.


\begin{cor}
Let $\lambda\in C_{cn}^{\ast}$.\\
(i) For every bounded sequence of maps $\phi_j\in B_n$, $j\in\N$, which converges pointwise to $\psi:I\to\C$ we have $\phi\in B_n$ and $\lambda_e\phi_j\to\lambda_e\psi$ as $j\to\infty$.

(ii) For every $\beta\in B_{nA}$,
$$
|\lambda_e\beta|\le|\lambda|_{C_{cn}^{\ast}}|\beta|.
$$
\end{cor}

\begin{proof}
1. Proof of (i). For every $k\in\{1,\ldots,n\}$, $\phi_{jk}\to\psi_k$ pointwise as $j\to\infty$. By Corollary 3.2, $\psi_k\in B$. Hence $\psi\in B_n$,
and for $j\to\infty$, again by Corollary 3.2,
$$
\lambda_e\phi_j=\sum_{k=1}^n(\lambda\circ E_k)_e\phi_{jk}\to\sum_{k=1}^n(\lambda\circ E_k)\psi_k=\lambda_e\psi.
$$

2. Proof of (ii). There exists a sequence of maps $\phi_j\in C_{cn}$, $j\in\N$, which converges pointwise to $\beta$ and satisfies
$|\phi_j|\le|\beta|$ for all $j\in\N$. Using assertion(i) we infer
$$
|\lambda_e\beta|=\lim_{j\to\infty}|\lambda\phi_j|\le|\lambda|_{C_{cn}^{\ast}}\sup_{j}|\phi_j|\le|\lambda|_{C_{cn}^{\ast}}|\beta|.
$$
\end{proof}

We proceed to extensions of maps $\Lambda\in L_c(C_{cn},Y)$ into a normed linear space $Y$
over the field $\C$ and begin with the map $\Lambda_0:B_{nA}\to\C^{Y^{\ast}}$ given by
$$
(\Lambda_0\phi)y^{\ast}=(y^{\ast}\circ\Lambda)_e\phi.
$$
The desired extension is the map $\Lambda_E:B_{nA}\to Y^{\ast\ast}$ from the next result.



\begin{prop}
Let $\Lambda\in L_c(C_{cn},Y)$ be given.

(i) For every $\phi\in B_{nA}$ the map $\Lambda_0\phi:Y^{\ast}\to\C$ is linear and continuous with
$$
|(\Lambda_0\phi)y^{\ast}|\le|y^{\ast}||\Lambda|_{L_c(C_{cn},Y)}|\phi|.
$$

(ii)
The map
$$
\Lambda_E:B_{nA}\ni\phi\mapsto\Lambda_0\phi\in Y^{\ast\ast}
$$
is linear and continuous with
$$
|\Lambda_E|_{L_c(B_{nA},Y^{\ast\ast})}\le|\Lambda|_{L_c(C_{cn},Y)}.
$$

(iii) For each $\phi\in C_{cn}$,
$$
\Lambda_E\phi=(\iota\circ\Lambda)\phi.
$$
\end{prop}

\begin{proof}
1. Proof of (i). Let $\phi\in B_{nA}$. The map $\Lambda_0\phi:Y^{\ast}\to\C$ is obviously linear, and for every $y^{\ast}\in Y^{\ast}$ we have
$$
|(\Lambda_0\phi)y^{\ast}|=|(y^{\ast}\circ\Lambda)_e\phi|\le|y^{\ast}\circ\Lambda|_{L_c(C_{cn},\C)}|\phi|\le|y^{\ast}||\Lambda|_{L_c(C_{cn},Y)}|\phi|
$$
because of Corollary 3.3 (ii). This yields continuity.

2. Proof of (ii). The map $\Lambda_E$ is obviously linear. The estimate from part (i) yields
$$
\sup_{\phi\in B_{nA}:|\phi|\le1}|\Lambda_E\phi|=\sup_{\phi\in B_{nA}:|\phi|\le1}\sup_{y^{\ast}\in Y^{\ast}:|y^{\ast}|\le1}|(\Lambda_E\phi)y^{\ast}|
\le|\Lambda|_{L_c(C_{cn},Y)},
$$
from which the assertions follow.

3. Proof of (iii). We have
$$
(\Lambda_E\phi)y^{\ast}=(y^{\ast}\circ\Lambda)_e\phi=y^{\ast}(\Lambda\phi)=(\iota(\Lambda\phi))y^{\ast}=((\iota\circ\Lambda)\phi)y^{\ast}
$$
for all $\phi\in C_{cn}$ and $y^{\ast}\in Y^{\ast}$.
\end{proof}

Next we consider curves and transformations induced by maps $\Lambda_E$.

\begin{prop}
Let $\Lambda\in L_c(C_{cn},Y)$ and intervals
$J\subset\R$ and $J_0\subset\R$ be given so that for all
$t\in J_0$ we have $[t-h,t]\subset J$. Let $v:J\to\C^n$ be
given.

(i) If $v$ is continuous then $v_t\in C_{cn}\subset B_{nA}$ for all
$t\in J_0$, and the map
$$
w:J_0\ni t\mapsto\Lambda_Ev_t\in Y^{\ast\ast}
$$
is continuous with values in $\iota Y$.

(ii) Suppose there is a sequence of continuous functions
$v_m:J\to\C^n$, $m\in\N$, which converges pointwise to $v$, and
$\sup_{m\in\N,t\in J}|v_m(t)|<\infty$. Then $v_t\in B_{nA}$ for all
$t\in J_0$, and the sequence of continuous functions
$$
w_m:J_0\ni t\mapsto\Lambda_Ev_{mt}\in Y^{\ast\ast},\,\,m\in\N,
$$
converges pointwise to
$$
w:J_0\ni t\mapsto\Lambda_Ev_t\in Y^{\ast\ast},
$$
with respect to the weak-star topology (of pointwise convergence) on $Y^{\ast\ast}$, and
$$
\sup_{m\in\N,t\in J_0}|w_m(t)|<\infty.
$$

(iii)
In case $\Lambda=\lambda\in C_{cn}^{\ast}$ we obtain that the map
$$
w:J_0\ni t\mapsto \lambda_ev_t\in\C
$$
has property (A).
\end{prop}

\begin{proof}
1. Proof of (i). The curve $J_0\ni t\mapsto v_t\in C_{cn}$ is continuous. Proposition 3.4 (iii) gives
$$
\Lambda_Ev_t=(\iota\circ\Lambda)v_t
$$
on $J_0$. Now continuity of $w$ and $w(J_0)\subset\iota Y$ are obvious.

2. Proof of(ii). Obviously, $v_t\in B_{nA}$ on $J_0$. Proposition 3.4 (ii) gives
$$
|w_m(t)|=|\Lambda_Ev_{mt}|\le|\Lambda|_{L_c(C_{cn},Y)}|v_{mt}|\le|\Lambda|_{L_c(C_{cn},Y)}\sup_{j\in\N,t\in J}|v_j(t)|<\infty
$$
for all $m\in\N$ and $t\in J_0$. Using Corollary 3.3 (i) we see that for each $t\in J_0$ and for all $y^{\ast}\in Y^{\ast}$ we have
$$
(\Lambda_Ev_{mt})y^{\ast}=(y^{\ast}\circ\Lambda)_ev_{mt}\to(y^{\ast}\circ\Lambda)_ev_t=(\Lambda_Ev_t)y^{\ast}.
$$
as $m\to\infty$.

3. Proof of (iii).  We proceed as in part 2. The curves $J_0\ni t\mapsto v_{mt}\in C_{cn}$, $m\in\N$, are continuous, and for such $m$ and $t\in J_0$,
 $\lambda_ev_{mt}=\lambda v_{mt}$ . It follows that the maps $w_m:J_0\ni t\mapsto \lambda_ev_{mt}\in\C$, $m\in\N$, are continuous.  Corollary 3.3 (ii) yields
$$
|w_m(t)|=|\lambda_ev_{mt}|\le|\lambda|_{C_{cn}^{\ast}}|v_{mt}|\le|\lambda|_{C_{cn}^{\ast}}\sup_{j\in\N,t\in J}|v_j(t)|<\infty
$$
for all $m\in\N$ and $t\in J_0$. For each $t\in J_0$ Corollary 3.3 (i) gives
$$
w_m(t)=\lambda_ev_{mt}\to\lambda_ev_t=w(t)\quad\text{as}\quad m\to\infty.
$$
\end{proof}

We may say that Proposition 3.5 (ii) transfers property
(A) from $v$ to $w$.


\begin{prop}
(i) (Approximation of piecewise continuous maps) Suppose
$v:J\to\C^n$, $J\subset\R$ an interval, has continuous restrictions
to $(-\infty,t)\cap J$ and to $[t,\infty)\cap J$, for some $t\in J$.
Then $v$ has property (A).

(ii) (Concatenation) Let $J$ and $K$ be adjacent intervals, $\max\,
J=a=\min \,K$. Suppose $v:J\to\C^n$ and $w:K\to\C^n$ with
$v(a)=w(a)$ both have property (A). Then also the concatenation
$vw:J\cup K\to\C^n$ given by $vw|J=v,vw|K=w$ has property (A).
\end{prop}

\begin{proof}
1. On (i). In case $t=\inf\,J$ there is nothing to show. Suppose
$\inf\,J<t$. For $m\in\N$ with $t-\frac{1}{m}\in J$ we define the
components of maps $v_m:J\to\C^n$ as follows. Set
\begin{eqnarray*}
v_{mk}(t)=v_k(t) &
\text{on} & \left(-\infty,t-\frac{1}{m}\right)\cap J,\\
v_{mk}(t)=v_k(t) & \text{on} & [t,\infty)\cap J,\\
\end{eqnarray*}
and
$$
v_{mk}(s)=v_k(t)+\,m\,\left(v_k(t)-v_k\left(t-\frac{1}{m}\right)\right)(s-t)
$$
for all $s\in J$ with $t-\frac{1}{m}<s<t$. The sequence of the continuous maps $v_m$
is uniformly bounded and converges pointwise to $v$.

2. On (ii). In case $J=\{a\}$, $vw=w$, and there is nothing to show.
Suppose $\inf\,J<a$. Consider a uniformly bounded sequence
$(v_m)_{m=1}^{\infty}$ of continuous maps which converges pointwise
to $v$, and a uniformly bounded sequence $(w_m)_{m=1}^{\infty}$ of
continuous maps which converges pointwise to $w$. For each $m\in\N$
with $\inf\,J <a-\frac{1}{m}$ choose a continuous map
$\hat{v}_m:J\to\C^n$ with $\hat{v}_m(t)=v_m(t)$ on
$(-\infty,a-\frac{1}{m}]\cap J$ and
$$
\hat{v}_m(t)=w_m(a)+\frac{w_m(a)-v_m\left(a-\frac{1}{m}\right)}{\frac{1}{m}}(t-a)\quad\text{on}\quad\left(a-\frac{1}{m},a\right].
$$
Then the sequence of the continuous maps $\hat{v}_m$ converges
pointwise to $v$ (at $t=a$, $\hat{v}_m(a)=w_m(a)\to w(a)=v(a)$ as $m\to\infty$), and it is uniformly bounded. It follows that the
continuous maps $\hat{v}_mw_m$ are uniformly bounded and converge
pointwise to $vw$.
\end{proof}
\section{Weak-star integrals and extended operators}


Let $J\subset\R$
be an interval, let $Y$ be a normed linear space  over the field
$\C$, and let a map $F:J\to Y^{\ast\ast}$ be given. Let $m$ denote the Lebesgue measure on $\R$. Suppose for every $y^{\ast}\in Y^{\ast}$ the function $J\ni t\mapsto F(t)y^{\ast}\in\C$ is Lebesgue-integrable, i. e., Lebesgue-measurable with
$\int_J|F(t)y^{\ast}|dm(t)<\infty$. Then we say that the weak-star integral $\int_J Fdm
=\int_J F(t)dm(t)\in\C^{Y^{\ast}}$ of $F$ exists, and define
$$
\left(\int_J Fdm\right)y^{\ast}=\int_J F(t)y^{\ast}dm(t).
$$
For subintervals $J_1\subset J$ we define $\int_{J_1} Fdm=\int_{J_1}(F|J_1)dm$.


\begin{prop}
Suppose the weak-star integral of $F:J\to Y^{\ast\ast}$ exists. Then it is a linear map $Y^{\ast}\to\C$, for every subinterval $J_1\subset J$ $\int_{J_1}Fdm$ exists, and for reals $a<b<c$ in $J$ we have
$$
\int_{(a,c)}Fdm=\int_{(a,b)}Fdm+\int_{(b,c)}Fdm.
$$
Also, for every $d\in\R$, the shifted map $J+d\ni t\mapsto F(t-d)\in Y^{\ast\ast}$ has a weak-star integral, with
$$
\int_{(a,b)}F(s)dm(s)=\int_{(a+d,b+d)}F(t-d)dm(t)
$$

\end{prop}

The proof is straightforward.

\begin{prop}
(Sufficient condition for integrability) Suppose $f:J\to\C^n$, with an interval $J\subset\R$, has property (A), and $\Lambda\in L_c(C_{cn},Y)$. Let a subinterval $J_0\subset J$ be given with $[t-h,t]\subset J$ for all $t\in J_0$, and
let reals $a<b$ in $J_0$ be given.  Then the weak-star integral
$$
\int_{(a,b)}\Lambda_Ef_tdm(t)
$$
exists and belongs to $Y^{\ast\ast}$, and
$$
\left|\int_{(a,b)}\Lambda_Ef_tdm(t)\right|\le(b-a)|\Lambda|_{L_c(C_{cn},Y)}\sup_{a<t<b}|f(t)|.
$$
\end{prop}

\begin{proof}
By Proposition 3.5 (ii), $f_t\in B_{nA}$ for all $t\in J_0$, and for each $y^{\ast}\in Y^{\ast}$ we have
$$
(\Lambda_Ef_t)y^{\ast}=(y^{\ast}\circ\Lambda)_ef_t\in\C.
$$
By Proposition 3.5 (iii) , applied to $y^{\ast}\circ\Lambda\in C_{cn}^{\ast}$, the map $(a,b)\ni t\mapsto(y^{\ast}\circ\Lambda)_ef_t\in\C$ has property (A). So it is a bounded Borel function, and therefore Lebesgue-integrable.
It follows that the weak-star integral $\int_{(a,b)}\Lambda_Ef_tdm(t):Y^{\ast}\to\C$  exists. It remains to show that this linear map maps bounded sets into bounded sets.  This follows from the estimate
$$
\left|\left(\int_{(a,b)}\Lambda_Ef_tdm(t)\right)y^{\ast}\right| = \left|\int_{(a,b)}(\Lambda_Ef_t)y^{\ast}dm(t)\right|= \left|\int_{(a,b)}(y^{\ast}\circ\Lambda)_ef_tdm(t)\right|
$$
$$
\le \int_{(a,b)}|(y^{\ast}\circ\Lambda)_ef_t|dm(t)
$$
$$
\le  (b-a)|y^{\ast}\circ\Lambda|_{C_{cn}^{\ast}}\sup_{a-h<t<b}|f(t)|\quad\text{(with Corollary 3.3 (ii))}
$$
$$
\le    (b-a)|y^{\ast}|_{Y^\ast}|\Lambda|_{L_c(C_{cn},Y)}\sup_{a-h<t<b}|f(t)|
$$
for all $y^{\ast}\in Y^{\ast}$.
\end{proof}

Lebesgue integration and extended continuous linear maps commute
in a sense which will be made precise. Below we derive such a result
for integrands in a special form as it will be needed in Section 6
and for further use.

Let reals $a<b$ be given, and a map $f:[a-h,b]\to\C^n$ with property (A) and with
$f(t)=0$ on $[a-h,a)$, and a continuous function $q:[a,b]\to\C$.
Then for every $t\in[a,b]$ all functions $[a,t]\ni s\mapsto
q(s)f_j(t-(s-a))\to\C$, $j\in\{1,\ldots,n\}$, have property (A).
So they are bounded Borel functions. Their restrictions to
intervals $(a,t)$, $a<t\le b$, are Lebesgue-integrable.
Therefore we have a map $g:[a-h,b]\to\C^n$ given by
$$
g(t)=\int_{(a,t)}q(s)f(t-s+a)dm(s)\quad\text{on}\quad(a,b]\quad\text{and}\quad g(t)=0\quad\text{on}\quad[a-h,a],
$$
with $g_j(t)=\int_{(a,t)}q(s)f_j(t-s+a)dm(s)$ for $a<t\le b$ and $j\in\{1,\ldots,n\}$.

\begin{prop}
$g$ is continuous.
\end{prop}

\begin{proof}
$g|[a-h,a]$ is constant, hence continuous. Boundedness of the integrand in the formula defining $g$ on $(a,b]$
yields that at $t=a$ $g$ is continuous from the right. In order to show continuity at $t_0\in(a,b]$ observe first that
for every $t\in(a,b]$ substitution yields
$$
g(t)=\int_{(a,t)}q(t-w+a)f(w)dm(w).
$$
For $a<u<t\le b$ we obtain
\begin{eqnarray*}
|g(t)-g(u)| & = &
\left|\int_{(u,t)}q(t-w+a)f(w)dm(w)\right.\\
& & \left.-\int_{(a,u)}(q(u-w+a)-q(t-w+a))f(w)dm(w)\right|\\
& \le & |t-u|\max_{a\le v\le t}|q(v)|\sup_{a\le w\le t}|f(w)|\\
& & +\int_{(a,u)}|q(u-w+a)-q(t-w+a)||f(w)|dm(w).
\end{eqnarray*}
This estimate, the uniform continuity of $q:[a,b]\to\C$, and the boundedness of $f$ can now be
used to get continuity of $g$ at $t_0\in(a,b]$.
\end{proof}

\begin{prop}
Let $\lambda\in C_{cn}^{\ast}$ and $a<t\le b$. All maps $q(s)f_{t-s+a}$, $a\le s\le t$, belong to $B_{nA}$, the map
$$
 (a,t)\ni s\mapsto\lambda_e(q(s)f_{t-s+a})\in\C
$$
has property (A), and  is Lebesgue-integrable.
\end{prop}

\begin{proof}
Let $a<t\le b$. As $f$ has property (A), Proposition 3.5 (ii) says that all segments $f_s$, $a\le s\le t$, belong to $B_{nA}$, and Proposition 3.5 (iii)
shows that the map  $(a,t)\ni s\mapsto\lambda_ef_s\in\C$ has property (A) as well. It follows that also the map
$(a,t)\ni s\mapsto\lambda_ef_{t-s+a}\in\C$ has property (A). Using multiplication by the bounded
continuous function $q$ and linearity one easily obtains the assertion.
\end{proof}

\begin{prop}
Let $\lambda\in C_{cn}^{\ast}$ and $a<t\le b$. Then
$$
\lambda\,g_t=\int_{(a,t)}\lambda_e(q(s)f_{t-s+a})dm(s).
$$
\end{prop}

\begin{proof}
1. Choose a sequence of continuous maps
$f_k:[a-h,b]\to\C^n$, $k\in\N$, which is uniformly bounded by
some $c\ge0$ and converges pointwise to $f$. Then for every $t\in(a,b]$ the bounded sequence of continuous maps $qf_k(t-\cdot+a:[a,t]\to\C^n$,
$k\in\N$, converges pointwise to $qf(t-\cdot+a):[a,t]\to\C^n$ as $k\to\infty$. Corollary 3.3 (i) yields that for $k\to\infty$ the bounded sequence of the continuous functions
$$
[a,t]\ni s\mapsto\lambda_ef_{k,t-s+a}\in\C,\quad k\in\N,
$$
converges pointwise to
$$
[a,t]\ni s\mapsto\lambda_ef_{t-s+a}\in\C.
$$
As $q$ is continuous and as $\lambda_e$  is linear we infer
that for $k\to\infty$  the continuous functions
$$
[a,t]\ni s\mapsto\lambda(q(s)f_{k,t-s+a})\in\C,\quad k\in\N,
$$
converge pointwise to
$$
[a,t]\ni s\mapsto\lambda_e(q(s)f_{t-s+a})\in\C.
$$
Theorem 2.1 (the  Lebesgue Dominated Convergence Theorem) yields
$$
\int_{(a,t)}\lambda_e(q(s)f_{t-s+a})dm(s)=\lim_{k\to\infty}\int_{(a,t)}\lambda(q(s)f_{k,t-s+a})dm(s).
$$

2. For every $k\in\N$ we have
$$
\int_{(a,t)}\lambda(q(s)f_{k,t-s+a})dm(s) = \int_a^t\lambda(q(s)f_{k,t-s+a})ds\quad\text{(Riemann-integral)}
$$
$$
=\lambda\left( \int_a^t q(s)f_{k,t-s+a}ds\right),
$$
and for all $u\in I$, with the linear continuous evaluation map
$$
ev_u:C_{cn}\ni\phi\mapsto\phi(u)\in\C^n,
$$
we have
\begin{eqnarray*}
 \left(\int_a^tq(s)f_{k,t-s+a}ds\right)(u) & = & ev_u \left(\int_a^tq(s)f_{k,t-s+a}ds\right)\\
= \int_a^tev_u(qf_{k,t-s+a})ds & = &   \int_a^tq(s)f_k(t-s+a+u)ds.
\end{eqnarray*}
In case $a<t+u$ the last integral is
$$
  \int_a^{t+u}q(s)f_k(t-s+a+u)ds+  \int_{t+u}^tq(s)f_k(t-s+a+u)ds.
$$

3. For $a<s<t$ and $u\in I$ with $t+u\le a$ we have $t-s+a+u\le t+u \le a$, and infer from pointwise convergence and from $f(v)=0$ for $a-h\le v\le a$ that
$$
q(s)f_k(t-s+a+u)\to0\quad\text{as}\quad k\to\infty.
$$
So in this case Theorem 2.1 yields
$$
\int_a^t q(s)f_k(t-s+a+u)ds\to0\quad\text{as}\quad k\to\infty.
$$
Next, for $u\in I$ with $a<t+u\,\, (\le t)$ we obtain
$$
\int_{t+u}^t q(s)f_k(t-s+a+u)ds\to0\quad\text{as}\quad k\to\infty
$$
since for $t+u<s<t$ we have $t-s+a+u\le a$, and $f_k(v)\to f(v)=0$ on $[a-h,a]$ as $k\to\infty$.

\smallskip

Also, for $u\in I$ with $a\le t+u$ we have
$$
\int_a^{t+u}q(s)f_k(t-s+a+u)ds\to\int_a^{t+u}q(s)f(t-s+a+u)ds=g(t+u)=g_t(u),
$$
due to Theorem 2.1.

\smallskip

Together,
$$
\int_a^tq(s)f_{k,t-s+a}ds\to   g_t\quad\text{pointwise on}\quad I \quad\text{as}\quad k\to\infty.
$$

4. Using Corollary 3.3 (i) we infer
$$
\int_a^t\lambda(qf_{k,t-s+a})ds=\lambda\left(\int_a^tq(s)f_{k,t-s+a})ds\right)\to\lambda g_t\quad\text{as}\quad k\to\infty.
$$
We combine this with the last equation in part 1 and arrive at
$$
\lambda g_t=\int_{(a,t)}\lambda_e(q(s)f_{t-s+a})dm(s).
$$
\end{proof}

\begin{cor}
Let $\Lambda\in L_c(C_{cn},Y)$ and $a<t\le b$. Then the weak-star integral of the map
$$
(a,t)\ni s\mapsto \Lambda_E(q(s)f_{t-s+a})\in Y^{\ast\ast}
$$
exists, with
$$
\iota(\Lambda g_t)=\int_{(a,t)}\Lambda_E(q(s)f_{t-s+a})dm(s).
$$
\end{cor}

\begin{proof}
Let $y^{\ast}\in Y^{\ast}$ be given. Proposition 4.4
shows that the map $(a,t)\ni s\mapsto(y^{\ast}\circ\Lambda)_e(q(s)f_{t-s+a})\in\C$ is Lebesgue-integrable. We have
$$
(y^{\ast}\circ\Lambda)_e(q(s)f_{t-s+a})=(\Lambda_E(q(s)f_{t-s+a}))y^{\ast}\quad\text{for all}\quad s\in(a,t).
$$
It follows that the weak-star integral of the map
$$
(a,t)\ni s\mapsto \Lambda_E(q(s)f_{t-s+a})\in Y^{\ast\ast}
$$
exists, and for every $y^{\ast}\in Y^{\ast}$ we have
\begin{eqnarray*}
\left(\int_{(a,t)} \Lambda_E(q(s)f_{t-s+a})dm(s)\right)y^{\ast} & = & \int_{(a,t)} \Lambda_E(q(s)f_{t-s+a}))y^{\ast}dm(s)\\
=\int_{(a,t)}(y^{\ast}\circ\Lambda)_e(q(s)f_{t-s+a})dm(s) & = &  (y^{\ast}\circ\Lambda)g_t\quad\text{(see Proposition 4.5)}\\
=y^{\ast}(\Lambda g_t) & = & (\iota(\Lambda g_t))y^{\ast}.
\end{eqnarray*}
\end{proof}

At the end of this section we formulate results on commutativity which look a bit more general and simpler than the preceding ones. Proofs use the same arguments as above.

Let reals $a<b$ and a map $x:[a-h,b]\to\C^n$ with property (A) be given, as well as $\lambda\in C_{cn}^{\ast}$ and $\Lambda\in L_c(C_{cn},Y)$. Then we have the curves
$$
X:[a,b]\ni t\mapsto x_t\in B_{nA},
$$
$$
\lambda_e\circ X:[a,b]\to\C\quad\text{and}\quad\Lambda_E\circ X:[a,b]\to Y^{\ast\ast}.
$$
We define
$$
z:[a-h,b]\to\C^n
$$
by
$$
z(t)=\int_{(a-h,t)}xdm.
$$

\begin{prop}
$z$ is continuous.
\end{prop}

\begin{prop}
$\lambda_e\circ X:[a,b]\to\C$ is Borel-measurable and bounded, hence Lebesgue-integrable.
\end{prop}

\begin{prop}
$\Lambda_E\circ X$ is bounded, and for each $y^{\ast}\in Y^{\ast}$ the function $[a,b]\ni t\mapsto ((\Lambda_E\circ X)(t))y^{\ast}\in\C$ has property (A), and is Lebesgue-integrable
\end{prop}

The preceding proposition says that the weak-star integral $\int\Lambda_E\circ X dm:Y^{\ast}\to\C$ exists.

\begin{prop}
Suppose in addition that $x(t)=0$ for $a-h\le t\le a$. Then for each $t\in(a,b]$
$$
\lambda z_t=\int_{(a,t)}\lambda_e\circ X dm.
$$
\end{prop}

\begin{prop}
For all $t\in(a,b]$,
$$
\iota(\Lambda z_t)=\int_{(a,t)}\Lambda_E\circ X dm.
$$
\end{prop}

The preceding equation implies
$$
\int_{(a,t)}\Lambda_E\circ X dm\in Y^{\ast\ast}
$$
for $a<t\le b$. For such $t$ and for $y^{\ast}\in Y^{\ast}$ we also get the estimate
$$
\left|\left(\int_{(a,t)}\Lambda_E\circ X dm\right)y^{\ast}\right|\le(t-a)|\Lambda|_{L_c(C_{cn},Y)}|y^{\ast}|\sup_{a-h<s<t}|x(s)|
$$
(which once more shows that $\int_{(a,t)}\Lambda_E\circ X dm\in Y^{\ast\ast}$).

\section{Linear neutral equations and discontinuous initial data}

Let $L\in L_c(C_{cn},\C^n)$ and $R\in L_c(C_{cn},\C^n)$ be given
and assume the hypothesis $(\Delta)$ from Section 1 is satisfied. Consider the initial value problem
\begin{eqnarray}
\frac{d}{dt}(v-L\circ V)(t) & = & Rv_t\quad\text{for}\quad t\ge0,\\
v_0 & = & \phi\in C_{cn},
\end{eqnarray}
with $V(t)=v_t$ as in Section 1. Solutions of the IVP (5.1)-(5.2) are understood as solutions of Eq. (1.2) which satisfy Eq. (5.2). The next result is well-known, see \cite{HVL,HM}. We include a short proof, which exploits the hypothesis $(\Delta)$, in order to keep this account self-contained.

\begin{prop}
Each $\phi\in C_{cn}$ uniquely determines a solution $v=v^{\phi}$ of the
IVP (5.1)-(5.2). Each linear map
$$
S(t):C_{cn}\ni\phi\mapsto v^{\phi}_t\in C_{cn},\quad t\ge0,
$$
is continuous.
\end{prop}

\begin{proof} 1. (Local solutions) Choose $T\in(0,\Delta)$ with $T|R|_{L_c(C_{cn},\C^n)}<1$. Let $C_T$ denote the Banach
space of continuous mappings $w:[-h,T]\to\C^n$ with $w_0=0$, normed
by $|w|_T=\max_{0\le t\le T}|w(t)|$. For $\chi\in C_{cn}$ we define the continuous extension
$\chi^c:[-h,\infty)\to\C^n$ by $\chi^c_0=\chi$ and
$\chi^c(t)=\chi(0)$ for $t>0$. For a
continuous function $v:[-h,T]\to\C^n$  with $v_0=\chi$ and for
$0\le t\le T$, Eq. (5.1) is equivalent to the equation
$$
v(t)-Lv_t=v(0)-Lv_0+\int_0^tRv_sds,
$$
or, using property ($\Delta$),
$$
v(t)=L\chi^c_t+\chi(0)-L\chi+\int_0^tRv_sds.
$$
For $w=v-\chi^c$ this yields
\begin{equation}
w(t)=L\chi^c_t+\chi(0)-L\chi-\chi^c(t)+\int_0^tRw_sds+\int_0^tR\chi^c_sds
\end{equation}
and $w_0=0$. Define
$$
A:C_{cn}\times C_T\to(\C^n)^{[-h,T]}
$$
by
\begin{eqnarray*}
A(\chi,w)(t) & = & 0\quad\text{for}\quad-h\le t\le 0,\\
A(\chi,w)(t) & = &
L\chi^c_t+\chi(0)-L\chi-\chi^c(t)+\int_0^tRw_sds+\int_0^tR\chi^c_sds\quad\text{for}\quad0<t\le
T.
\end{eqnarray*}
Then $A(C_{cn}\times C_T)\subset C_T$, and the linear map $A$ is
continuous. For all $\chi\in C_{cn}$, for $w,w_1$ in $C_T$, and for $0\le
t\le T$,
$$
|A(\chi,w)(t)-A(\chi,w_1)(t)|=\left|\int_0^tRw_sds-\int_0^tRw_{1,s}ds\right|\le
T|R|_{L_c(C_{cn},\C^n)}|w-w_1|_T.
$$
It follows that each map $A(\chi,\cdot):C_T\to C_T$, $\chi\in C_{cn}$,
is a contraction, with a unique fixed point $w^{\chi}\in C_T$. The
linear map
$$
C_{cn}\ni\chi\mapsto w^{\chi}\in C_T
$$
is continuous, due to the estimate
$$
|w^{\chi}|_T=|A(\chi,w^{\chi})|_T\le(2|L|_{L_c(C_{cn},\C^n)}+2+T|R|_{L_c(C_{cn},\C^n)})|\chi|+T|R|_{L_c(C_{cn},\C^n)}|w^{\chi}|_T
$$
which gives
$$
|w^{\chi}|_T\le\frac{2|L|_{L_c(C_{cn},\C^n)}+2+T|R|_{L_c(C_{cn},\C^n)}}{1-T|R|_{L_c(C_{cn},\C^n)}}|\chi|.
$$
Using the equation
$$
v^{1,\chi}(t)=w^{\chi}(t)+\chi^c(t)\quad\text{for}-h\le t\le T
$$
we obtain a continuous map $v^{1,\chi}:[-h,T]\to\C^n$ such that the map
$$
v^{1,\chi}-L\circ V^{1,\chi}:[0,T]\ni t\mapsto v^{1,\chi}(t)-Lv^{1,\chi}_t\in\C^n
$$
is differentiable and satisfies Eq. (5.1) on $[0,T]$.

Also, $v^{1,\chi}$ is uniquely determined by $\chi$, and the linear map
$$
C_{cn}\ni\chi\mapsto v^{1,\chi}\in C_T
$$
is continuous.

2. Let $\phi\in C_{cn}$ be given. Using the result of part 1 we get a sequence of
continuous solutions $v^{(j)}:[-h,T]\to\C^n$ of Eq. (5.1) in the sense of part 1, with
$v^{(1)}=v^{1,\phi}$, and $v^{(j+1)}_0=v^{(j)}_T$ for all $j\in\N$. The
equations $v(t)=v^{(j)}(t-(j-1)T)$ for $(j-1)T\le t\le jT$,
$j\in\N$, define the solution $v=v^{\phi}$, and the proof is easily
completed.
\end{proof}

The maps $S(t)$, $t\ge0$, of Proposition 5.1 constitute a strongly continuous semigroup on the space $C_{cn}$.

In the sequel we consider the extensions $L_E:B_{nA}\to(\C^n)^{\ast\ast}$ and
$R_E:B_{n,A}\to(\C^n)^{\ast\ast}$ of $L$ and $R$, respectively, and discuss
solutions of an integrated version of the IVP (5.1)- (5.2.

\begin{prop}
Suppose $v:[-h,\infty)\to\C^n$ is a solution of Eq. (5.1). Then all restrictions of the map
$$
[0,\infty)\ni s\mapsto R_Ev_s\in(\C^n)^{\ast\ast}
$$
to bounded intervals have a weak-star integral, with
\begin{equation}
\iota(v(t)) -L_Ev_t=\iota(v(0))-L_Ev_0+\int_{(0,t)}R_Ev_sdm(s)
\end{equation}
for all $t>0$. 
\end{prop}

\begin{proof}
The maps $v$, $V:[0,\infty)\ni t\mapsto v_t\in C_{cn}$ and $R\circ V$ are continuous. As the right hand side of Eq. (5.1) is given by a continuous
map we can integrate this equation and find
$$
v(t)-Lv_t=v(0)-Lv_0+\int_{(0,t)}Rv_sdm(s)\quad\text{for every}\quad t>0.
$$
Using Proposition 3.4 (iii) we infer
\begin{eqnarray*}
\iota(v(t))-L_Ev_t & = & \iota(v(t))-\iota(Lv_t)=\iota(v(0))-\iota(Lv_0)+\iota\left(\int_{(0,t)}Rv_sdm(s)\right)\\
& = & \iota(v(0))-L_Ev_0+\iota\left(\int_{(0,t)}Rv_sdm(s)\right)
\end{eqnarray*}
for such $t$. Finally, use
\begin{eqnarray*}
\left(\iota\left(\int_{(0,t)}Rv_sdm(s)\right)\right)y^{\ast} & = & y^{\ast}\left(\int_{(0,t)}Rv_sdm(s)\right)\\
& = &  y^{\ast}\left(\int_0^tRv_sds\right)\quad\text{(Riemann integral)}\\
& = &  \int_0^ty^{\ast}(Rv_s)ds =   \int_0^t(\iota(Rv_s))y^{\ast}ds\\
& = &  \int_0^t(R_Ev_s)y^{\ast}ds\quad\text{(with Proposition 3.4 (iii))}\\
& = &   \int_{(0,t)}(R_Ev_s)y^{\ast}dm(s)=  \left(\int_{(0,t)}R_Ev_sdm(s)\right)y^{\ast}
\end{eqnarray*}
for each $y^{\ast}\in(C^n)^{\ast}$.
\end{proof}

Set
$$
\Delta_e=\frac{\Delta}{2}.
$$

\begin{prop}
For each $\phi\in B_{nA}$ with $\phi(t)=0$ on $[-h,-\Delta_e]$,
$L_E\phi=0$.
\end{prop}

\begin{proof}
Choose a bounded sequence of functions $\phi_j\in C_{cn}$ which
converges pointwise to $\phi$, and choose $\chi\in C_c$ with
$\chi(t)=0$ on $[-h,-\Delta]$ and $\chi(t)=1$ on $[-\Delta_e,0]$.
Then the functions $\chi\phi_j\in C_{cn}$ are uniformly bounded
and converge pointwise to $\phi$.  By Proposition 3.4 (iii) and by hypothesis $(\Delta)$,
$L_E\chi\phi_j=\iota(L\chi\phi_j)=\iota0=0$ for all $j$. Using Proposition 3.5 (ii)  with  $J=I$ and $J_0=\{0\}$ we infer
$$
(L_E\phi)y^{\ast}=\lim_{j\to\infty}(L_E\chi\phi_j)y^{\ast}=0
$$
for all $y^{\ast}\in(\C^n)^{\ast}$.
\end{proof}

Proposition 5.3 will be used in the proof of Proposition 5.7  below.


\begin{prop}
Let $\phi\in B_{nA}$ be given. Then the continuation $\phi^c:\R\to\C^n$ defined by  $\phi^c(t)=\phi(-h)$ for
$t<-h$ and $\phi^c(t)=\phi(0)$ for $t>0$ has property (A), all segments $(\phi^c)_t$, $t\in\R$, belong to $B_{nA}$,
and the map
$$
\iota^{-1}\circ L_E\circ\Phi^c\,\,\text{with}\,\,\Phi^c:\R\ni t\mapsto
(\phi^c)_t\in B_{nA}
$$
has property (A). Each segment of this map belongs to
$B_{nA}$.
\end{prop}

\begin{proof}
Choose a bounded sequence of maps  $\phi_j\in C_{cn}$ which
converges pointwise to $\phi$. Then the extensions $(\phi_j)^c$
are continuous and uniformly bounded and converge pointwise to
$\phi^c$. So $\phi^c$ has property (A). This implies
that each segment $(\phi^c)_t$, $t\in\R$, has property (A), and
belongs to $B_{nA}$. Proposition 3.5 (ii) yields that the map
$$
\R\ni t\mapsto L_E\circ\Phi^c\in(\C^n)^{\ast\ast}
$$ 
has property (A), with respect to the weak-star topology on $(\C^n)^{\ast\ast}$, which coincides with the norm topology. As $\iota$ is a norm-preserving isomorphism here we infer easily that the map $\iota^{-1}\circ L_E\circ\Phi^c$ has property (A) as well. 
\end{proof}

\begin{prop}
(An auxiliary integral equation) (i) Let $t_0>0$, $\phi\in B_{nA}$, and let a continuous map $w:[0,t_0]\to\C^n$ be given with
$$
\iota(w(0))=\iota(\phi(0))-L_E\phi.
$$ 
Then the map $w^{\phi^c}:(-\infty,t_0]\to\C^n$ given by 
\begin{eqnarray*}
w^{\phi^c}(t) & = & w(t)\quad\text{on}\quad[0,t_0],\\
w^{\phi^c}(t) & = & \phi^c(t)-(\iota^{-1}\circ L_E\circ\Phi^c)(t)\quad\text{on}\quad(-\infty,0],
\end{eqnarray*}
has property (A). All segments $(w^{\phi^c})_s$,  $s\le t_0$, and  $(\iota^{-1}\circ L_E\circ\Phi^c)_s$, $s\in\R$, belong to $B_{nA}$.
For all $t\in(0,t_0]$ the weak-star integrals
$$
\int_{(0,t)}R_E(w^{\phi^c})_sdm(s)\quad\text{and}\quad\int_{(0,t)}R_E(\iota^{-1}\circ L_E\circ\Phi^c)_sdm(s)
$$
exist.

(ii) There exists
$t_0\in(0,\Delta_e]$ such that for every $\phi\in B_{nA}$ there is a
uniquely determined continuous map $w:[0,t_0]\to\C^n$
with $\iota(w(0))=\iota(\phi(0))-L_E\phi$ such that for all $t\in(0,t_0]$  we have
\begin{equation}
\iota(w(t))=\iota(\phi(0))-L_E\phi+\int_{(0,t)}R_E(w^{\phi^c})_sdm(s)+\int_{(0,t)}R_E(\iota^{-1}\circ L_E\circ\Phi^c)_sdm(s).
\end{equation}

(iii) (Uniform bound) There exists $c_0\ge 0$ with
$$
\max_{0\le t\le t_0}|w(\phi)(t)|\le c_0|\phi|
$$
for all $\phi\in B_{nA}$, with $w(\phi)=w$ from (i).
\end{prop}

\begin{proof}
1. Proof of (i). Proposition 5.4 yields that $\phi^c$ and $\iota^{-1}\circ L_E\circ\Phi^c$ have property (A). Their difference also has property (A). Using this and the continuity of $w$ and Proposition 3.6 (ii) one finds that $w^{\phi^c}$ has property (A). Then all segments $(w^{\phi^c})_s$, $s\le t_0$, and $(\iota^{-1}\circ L_E\circ\Phi^c)_s$, $s\in\R$, have property (A) and belong to $B_{nA}$.
Proposition 4.2  shows that for $0<t\le t_0$ the weak-star integrals
 $$
\int_{(0,t)}R_E(w^{\phi^c})_sdm(s)\quad\text{and}\quad\int_{(0,t)}R_E(\iota^{-1}\circ L_E\circ\Phi^c)_sdm(s)
$$
exist and belong to $(\C^n)^{\ast\ast}$.

2. Proof of (ii).  Choose $t_0\in(0,\Delta_e]$ with $t_0|R|_{L_c(C_{cn},\C^n)}<1$. Let $W$
denote the Banach space of continuous maps $w:[0,t_0]\to\C^n$, with
$|w|=\max_{0\le t\le t_0}|w(t)|$. For $\phi\in B_{nA}$
consider the closed subset $W_{\phi}\subset W$ given by the equation
$\iota(w(0))=\iota(\phi(0))-L_E\phi$. For $w\in W_{\phi}$ the map $K_{\phi}w:[0,t_0]\to\C^n$ given by $\iota(K_{\phi}w)(0)=\iota(\phi(0))-L_E\phi$ and
$$
\iota((K_{\phi}w)(t))=\iota(\phi(0))-L_E\phi+\int_{(0,t)}R_E(w^{\phi^c})_sdm(s)+\int_{(0,t)}R_E(\iota^{-1}\circ L_e\circ\Phi^c)_sdm(s)
$$
for $0<t\le t_0$ is continuous, due to the fact that $\iota$ is norm-preserving and to the estimate 
$$
\left|\int_{(u,t)}R_E(w^{\phi^c})_sdm(s)\right| + \left|\int_{(u,t)}R_E(\iota^{-1}\circ L_e\circ\Phi^c)_sdm(s)\right|
$$
$$
=\sup_{|y^{\ast}|\le 1} \left|\left(\int_{(u,t)}R_E(w^{\phi^c})_sdm(s)\right)y^{\ast}\right|+\sup_{|y^{\ast}|\le 1}\left|\left (\int_{(u,t)}R_E(\iota^{-1}\circ L_e\circ\Phi^c)_sdm(s)\right)y^{\ast}\right|
$$
$$
\le \sup_{|y^{\ast}|\le 1} \left|\int_{(u,t)}(R_E(w^{\phi^c})_s)y^{\ast}dm(s)\right|+\sup_{|y^{\ast}|\le 1}\left|\int_{(u,t)}(R_E(\iota^{-1}\circ L_e\circ\Phi^c)_s)y^{\ast}dm(s)\right|
$$
$$
\le  \sup_{|y^{\ast}|\le 1} \int_{(u,t)}|(R_E(w^{\phi^c})_s)y^{\ast}|dm(s)+\sup_{|y^{\ast}|\le 1}\int_{(u,t)}|R_E(\iota^{-1}\circ L_e\circ\Phi^c)_s)y^{\ast}|dm(s)
$$
$$
\le   \int_{(u,t)}|R_E(w^{\phi^c})_s|dm(s)+\int_{(u,t)}|R_E(\iota^{-1}\circ L_e\circ\Phi^c)_s|dm(s)
$$
$$
\le \int_{(u,t)}|R|_{L_c(C_{cn},\C^n)}|(w^{\phi^c})_s|dm(s)
$$
$$
+\int_{(u,t)}|R|_{L_c(C_{cn},\C^n)}|(\iota^{-1}\circ L_e\circ\Phi^c)_s|dm(s)\quad\text{(with Proposition 3.4 (ii))}
$$
$$
\le(t-u)|R|_{L_c(C_{cn},\C^n)}|(|w|+|\phi|+|L|_{L_c(C_{cn},\C^n)}|\phi|)
$$
$$
+(t-u)|R|_{L_c(C_{cn},\C^n)}|L|_{L_c(C_{cn},\C^n)}|\phi|
\quad\text{for}\quad 0<u<t\le t_0.
$$
The map 
$$
K_{\phi}:W_{\phi}\ni w\mapsto K_{\phi}w\in W_{\phi}
$$
 is a contraction
since for all $w,\hat{w}$ in $W_{\phi}$ and for all $t\in(0,t_0]$ we
have
\begin{eqnarray*}
|(K_{\phi}w)(t)-(K_{\phi}\hat{w})(t)| & = & \left|\int_{(0,t)}R_E((w^{\phi^c})_s-(\hat{w}^{\phi^c})_s)dm(s)\right|\\
& = &  sup_{|y^{\ast}|\le1} \left|\left(\int_{(0,t)}R_E(w^{\phi^c})_s-(\hat{w}^{\phi^c})_s)dm(s)\right)y^{\ast}\right|\\
& = &  sup_{|y^{\ast}|\le1} \left|\int_{(0,t)}(R_E(w^{\phi^c})_s-(\hat{w}^{\phi^c})_s))y^{\ast}dm(s)\right|\\
& \le & t_0|R|_{L_c(C_{cn},\C^n)}\sup_{0\le s\le t_0}|(w^{\phi^c})_s-(\hat{w}^{\phi^c})_s|\\
& &  \text{(with Proposition 3.4)}\\
& \le & t_0|R|_{L_c(C_{cn},\C^n)}|w-\hat{w}|
\end{eqnarray*}
(with $w^{\phi^c}(u)=\hat{w}^{\phi^c}(u)$ for $u<0$, and $w^{\phi^c}(u)=w(u)$ on $[0,t_0]$, and
$\hat{w}^{\phi^c}(u)=\hat{w}(u)$ on $[0,t_0]$). The unique fixed point $w=w(\phi)$ of $K_{\phi}$ is the desired solution of the integral equation (5.5). In order to obtain uniqueness observe that each continuous map $\hat{w}:[0,t_0]\to\C^n$ which satisfies $\iota(\hat{w}(0))=\iota(\phi(0))-L_E\phi$ and Eq. (5.5) on $(0,t_0]$ belongs to $W_{\phi}$ and is indeed a fixed point of $K_{\phi}$, so $\hat{w}=w(\phi)$.

3. Proof of (iii). Let $\phi\in B_{nA}$ be given. Set $w=w(\phi)$ and $w^c=w^{\phi^c}$. By continuity there exists
$t\in[0,t_0]$ with
$$
|w(t)|=|w|.
$$
In case $t=0$ Proposition 3.4 (ii) and the fact that $\iota$ preserves 
the norm combined yield $|w(t)|\le(1+|L|)|\phi|$. In case $0<t\le t_0$ we have
\begin{eqnarray*}
|w| & = & |w(t)| =|\iota(w(t))|\le|\iota(\phi(0))|+|L_E\phi|+\sup_{|y^{\ast}|\le 1} \left|\left(\int_{(0,t)}R_E(w^c)_sdm(s)\right)y^{\ast}\right|\\
& & +\sup_{|y^{\ast}|\le 1}\left|\left (\int_{0,t)}R_E(\iota^{-1}\circ L_e\circ\Phi^c)_sdm(s)\right)y^{\ast}\right|\\
& \le & |\phi|+|L|_{L_c(C_{cn},\C^n)}|\phi|+\sup_{|y^{\ast}|\le 1} \left|\int_{(0,t)}(R_E(w^c)_s)y^{\ast}dm(s)\right|\\
& & +\sup_{|y^{\ast}|\le 1}\left|\int_{0,t)}(R_E(\iota^{-1}\circ L_e\circ\Phi^c)_s)y^{\ast}dm(s)\right|\\
& \le & (1+|L|_{L_c(C_{cn},\C^n)})|\phi|+t_0|R|_{L_c(C_{cn},\C^n)}(|w|+|\phi|+|L|_{L_c(C_{cn},\C^n)}|\phi|)\\
& & + t_0|R|_{L_c(C_{cn},\C^n)}|L|_{L_c(C_{cn},\C^n)}|\phi|\\
& = & |\phi|(1+|L|_{L_c(C_{cn},\C^n)}+t_0|R|_{L_c(C_{cn},\C^n)}(1+|L|_{L_c(C_{cn},\C^n)})\\
& & +t_0|R|_{L_c(C_{cn},\C^n)}|L|_{L_c(C_{cn},\C^n)})+t_0|R|_{L_c(C_{cn},\C^n)}|w|.
\end{eqnarray*}
It follows that
\begin{eqnarray*}
(1-t_0|R|_{L_c(C_{cn},\C^n)})|w| & \le  & (1+|L|_{L_c(C_{cn},\C^n)}\\
& & +t_0|R|_{L_c(C_{cn},\C^n)}(1+2|L|_{L_c(C_{cn},\C^n)}))|\phi|.
\end{eqnarray*}
\end{proof}

\begin{prop}
(Continuous dependence on initial data with respect to pointwise
convergence) Suppose the maps $\phi_j\in B_{nA}$, $j\in\N$, are
uniformly bounded and converge pointwise to $\phi\in B_{nA}$. Then
$$
w(\phi_j)\to
w(\phi)\,\,\text{uniformly}\,\,\text{as}\,\,j\to\infty.
$$
\end{prop}

\begin{proof}
1. Consider $t_0$ and $c_0$ as in Proposition 5.5.
Let $c=\sup_{j,t}|\phi_j(t)|<\infty$. Then also
$$
\sup_{j,t}|(\phi_j)^c(t)|\le c,
$$
and for each $t\in\R$ we have
$$
(\phi_j)^c(t)\to\phi^c(t)\,\,\text{as}\,\,j\to\infty.
$$
It follows that for each $s\in\R$ the segments $((\phi_j)^c)_s$
are uniformly bounded by $c$ and converge pointwise to
$(\phi^c)_s$ as $j\to\infty$. Due to Proposition 5.4 all these
segments belong to $B_{nA}$.

2. Using Proposition 3.5 (ii) we infer that for each $s\in\R$ we have
$$
(L_E\circ\Phi_j^c)(s)=L_E((\phi_j)^c)_s\to
L_E(\phi^c)_s=(L_E\circ\Phi^c)(s)\,\,\text{as}\,\,j\to\infty,
$$
with respect the weak-star topology on $(\C^n)^{\ast\ast}$, which coincides with the norm-topology. 
By Proposition 3.4,
$$
|(L_E\circ\Phi_j^c)(s)| \le  |L|_{L_c(C_{cn},\C^n)}|((\phi_j)^c)_s|\le
|L|_{L_c(C_{cn},\C^n)}|c\,\,\text{for all}\,\,j\in\N,\,s\in\R.
$$

3. Consequently, for each $s\in\R$,
$$
(\iota^{-1}\circ L_E\circ\Phi_j^c)_s\to
(\iota^{-1}\circ L_E\circ\Phi^c)_s\,\,\text{pointwise
on}\,\,I\,\,\text{as}\,\,j\to\infty
$$
and
$$
|(\iota^{-1}\circ L_E\circ\Phi_j^c)_s|\le |L|_{L_c(C_{cn},\C^n)}c\,\,\text{for
all}\,\,j\in\N,\,s\in\R.
$$

4. Let $s\in\R$. Proposition 3.5 (ii), applied to $J=I$ and $J_0=\{0\}$, yields
$$
(R_E(\iota^{-1}\circ L_E\circ\Phi_j^c)_s)y^{\ast}\to
(R_E(\iota^{-1}\circ L_E\circ\Phi^c)_s)y^{\ast}\,\,\text{as}\,\,j\to\infty,
$$
for each $y^{\ast}\in(\C^n)^{\ast}$. Also, by Proposition 3.4 and  by part 3,
$$
|R_E(\iota^{-1}\circ L_E\circ\Phi_j^c)_s|=\sup_{|y^{\ast}|\le1}|(R_E(\iota^{-1}\circ L_E\circ\Phi_j^c)_s)y^{\ast}|
$$
$$
\le |R|_{L_c(C_{cn},\C^n)}|(\iota^{-1}\circ L_E\circ\Phi_j^c)_s|\le |R|_{L_c(C_{cn},\C^n)} |L|_{L_c(C_{cn},\C^n)}c
$$
for all $j\in\N$, $s\in\R$.

5. Let $j\in\N$. Set $w_j=w(\phi_j)$ and $w_j^c=w_j^{\phi_j^c}$. Recall
$w_j^c(t)=(\phi_j)^c(t)-\iota^{-1}(L_E((\phi_j)^c)_t)$
for $t<0$ and $w_j^c(t)=w_j(t)$ for
$0\le t\le t_0$. By Proposition 5.5 (iii),
$$
|w_j^c(t)|=|w_j(t)|\le c_0|\phi_j|\le c_0c\,\,\text{for
all}\,\,t\in[0,t_0]
$$
and
$$
|w_j^c(t)|\le|\phi_j|+ |L|_{L_c(C_{cn},\C^n)}|\phi_j|\,\,\text{for}\,\,t<0.
$$
Combining these estimates we infer
$$
|w_j^c(t)|\le c_0c+c+ |L|_{L_c(C_{cn},\C^n)}c\,\,\text{for
all}\,\,j\in\N,\,t\le t_0.
$$

6. Next, Proposition 3.4 yields
$$
|R_E(w_j^c)_s|\le  |R|_{L_c(C_{cn},\C^n)} |(w_j^c)_s|\le
 |R|_{L_c(C_{cn},\C^n)}c(c_0+1+ |L|_{L_c(C_{cn},\C^n)})
$$
for all $j\in\N$, $s\in[0,t_0]$.

7. Observe that due to part 2 for each $s\le0$ we have
$$
w_j^c(s)=(\phi_j)^c(s)-\iota^{-1}(L_E((\phi_j)^c)_s)\to\phi^c(s)-\iota^{-1}(L_E(\phi^c)_s)
$$
as $j\to\infty$.\\

8. Recall the Banach space $W$ from the proof of Proposition 5.5. The sequence of the maps $w_j\in W$ is bounded, see part 5. 
It also is equicontinuous as for $u<t$ in $(0,t_0]$  we have
$$
|w_j(t)-w_j(u)|= |\iota(w_j(t))-\iota(w_j(u))|
$$
$$
\le\left|\int_{(0,t)}R_E(w_j^c)_sdm(s)-\int_{(0,u)}R_E(w_j^c)_sdm(s)\right|
$$
$$
+\left|\int_{(0,t)}R_E(\iota^{-1}\circ L_E\circ\Phi_j^c)_sdm(s)-\int_{(0,u)}R_E(\iota^{-1}\circ L_E\circ\Phi_j^c)_sdm(s)\right|
$$
$$
 = \sup_{|y^{\ast}|\le1}\left|\left(\int_{(0,t)}R_E(w_j^c)_sdm(s)-\int_{(0,u)}R_E(w_j^c)_sdm(s)\right)y^{\ast}\right|
$$
$$
+ \sup_{|y^{\ast}|\le1}\left|\left(\int_{(0,t)}R_E(\iota^{-1}\circ L_E\circ\Phi_j^c)_sdm(s)-
\int_{(0,u)}R_E(\iota^{-1}\circ L_E\circ\Phi_j^c)_sdm(s)\right)y^{\ast}\right|
$$
$$
=\sup_{|y^{\ast}\le1}\left|\left(\int_{(0,t)}R_E(w_j^c)_sdm(s)\right)y^{\ast}-\left(\int_{(0,u)}R_E(w_j^c)_sdm(s)\right)y^{\ast}\right|
$$
$$
+ \sup_{|y^{\ast}|\le1}\left|\left(\int_{(0,t)}R_E(\iota^{-1}\circ L_E\circ\Phi_j^c)_sdm(s)\right)y^{\ast}\right.
$$
$$
-\left.\left(\int_{(0,u)}R_E(\iota^{-1}\circ L_E\circ\Phi_j^c)_sdm(s)\right)y^{\ast}\right|
$$
$$
 = \sup_{|y^{\ast}|\le1}\left|\int_{(0,t)}(R_E(w_j^c)_s)y^{\ast}dm(s)-\int_{(0,u)}(R_E(w_j^c)_s)y^{\ast}dm(s)\right|
$$
$$
+ \sup_{|y^{\ast}|\le1}\left|\int_{(0,t)}(R_E(\iota^{-1}\circ L_E\circ\Phi_j^c)_s)y^{\ast}dm(s)-\int_{(0,u)}(R_E(\iota^{-1}\circ L_E\circ\Phi_j^c)_s)y^{\ast}dm(s)\right|
$$
$$= \sup_{|y^{\ast}|\le1}\left|\int_{(u,t)}(R_E(w_j^c)_s)y^{\ast}dm(s)\right|+ \sup_{|y^{\ast}|\le1}\left|\int_{(u,t)}(R_E(\iota^{-1}\circ L_E\circ\Phi_j^c)_s)y^{\ast}dm(s)\right|
$$
$$
\le  (t-u) |R|_{L_c(C_{cn},\C^n)}(c_0c+c+ |L|_{L_c(C_{cn},\C^n)}c)+(t-u) |R|_{L_c(C_{cn},\C^n)} |L|_{L_c(C_{cn},\C^n)}c,
$$
and the same estimate holds for $0=u<t\le t_0$, due to continuity.
Now the Theorem of Ascoli and Arzel\`a yields a subsequence of maps
$w_{j_k}$, $k\in\N$, which is uniformly convergent to some
continuous map $\hat{w}:[0,t_0]\to\C^n$. Recall the set $W_{\phi}$
from the proof of Proposition 5.5. We have $\hat{w}\in W_{\phi}$
since
$$
\iota(\hat{w}(0))=\iota(\lim_{k\to\infty}w_{j_k}(0))=\lim_{k\to\infty}\iota(\phi_{j_k}(0)-\iota^{-1}(L_E\phi_{j_k}))=\iota(\phi(0))-L_E\phi,
$$
see part 2 for $s=0$. Next we prove that $\hat{w}=w(\phi)$.

9. We define $\hat{w}^c=\hat{w}^{\phi^c}$, so
\begin{eqnarray*}
\hat{w}^c(t) & = & \hat{w}(t)\,\,\text{for}\,\,0\le t\le t_0,\\
\hat{w}^c(t) & = &
\phi^c(t)-(\iota^{-1}\circ L_E\circ\Phi^c)(t)\,\,\text{for}\,\,t<0.
\end{eqnarray*}
Using part 7 and the convergence of the maps $w_{j_k}$ to
$\hat{w}$ we infer that for each $s\le t_0$
$$
w_{j_k}^c(s)\to\hat{w}^c(s)\,\,\text{as}\,\,k\to\infty.
$$
It follows that for each $s\le t_0$ the segments
$(w_{j_k}^c)_s\in B_{nA}$, $k\in\N$, converge
pointwise to $(\hat{w}^c)_s$ as $k\to\infty$. Using the bound
from part 5 and Proposition 3.5 (ii) (with $J=I$ and $J_0=\{0\}$) we infer 
that for each $s\le t_0$
$$
R_E(w_{j_k}^c)_s\to
R_E(\hat{w}^c)_s\,\,\text{as}\,\,k\to\infty,
$$
with respect to the weak-star topology on $(\C^n)^{\ast\ast}$. That is, for every $y^{\ast}\in (C^n)^{\ast}$ and for each $s\le t_0$,
$$
(R_E(w_{j_k}^c)_s)y^{\ast}\to
(R_E(\hat{w}^c)_s)y^{\ast}\,\,\text{as}\,\,k\to\infty.
$$
Using part 5 for $k\in\N$ and $s\le t_0$ we get the estimate
\begin{eqnarray*}
|(R_E(w_{j_k}^c)_s)y^{\ast}| & \le &  |R|_{L_c(C_{cn},\C^n)}|(w_{j_k}^c)_s||y^{\ast}|\\
&  \le & |R|_{L_c(C_{cn},\C^n)}( c_0c+c+ |L|_{L_c(C_{cn},\C^n)}c)|y^{\ast}|
\end{eqnarray*}
for every $y^{\ast}\in(\C^n)^{\ast}$.

Let $t\in(0,t_0]$. Proposition 5.5 (i) (existence of weak-star integrals) and the Lebesgue Dominated Convergence Theorem combined yield
$$
\int_{(0,t)}(R_E(w_{j_k}^c)_s)y^{\ast}dm(s)\to\int_{(0,t)}(R_E(\hat{w}^c)_s)y^{\ast}dm(s)\,\,\text{as}\,\,k\to\infty,
$$
for every $y^{\ast}\in(\C^n)^{\ast}$.

Similarly we obtain for each  $y^{\ast}\in(\C^n)^{\ast}$ that
$$
\int_{(0,t)}(R_E(\iota^{-1}\circ L_E\circ\Phi_{j_k}^c)_s)y^{\ast}dm(s)\to\int_{(0,t)}(R_E(\iota^{-1}\circ L_E\circ\Phi^c)_s)y^{\ast}dm(s)\,\,\text{as}\,\,k\to\infty.
$$
From parts 1 and 2 for $s=0$ we get
$$
\phi_{j_k}(0)\to\phi(0)\,\,\text{and}\,\,\iota^{-1}(L_E\phi_{j_k})\to\iota^{-1}L_E\phi\,\,\text{as}\,\,k\to\infty.
$$
Recall Eq. (5.5). For every  $y^{\ast}\in(\C^n)^{\ast}$ and $k\in\N$ we have
\begin{eqnarray*}
(\iota(w_{j_k}(t)))y^{\ast} & = & (\iota(\phi_{j_k}(0)))y^{\ast}-(L_E\phi_{j_k})y^{\ast}+\left(\int_{(0,t)}R_E(w^{\phi^c})_sdm(s)\right)y^{\ast}\\
& & +\left(\int_{(0,t)}R_E(\iota^{-1}\circ L_E\circ\Phi^c)_sdm(s)\right)y^{\ast}\\
& = & y^{\ast}(\phi_{j_k}(0))-(L_E\phi_{j_k})y^{\ast}+\int_{(0,t)}(R_E(w^{\phi^c})_s)y^{\ast}dm(s)\\
&  & +\int_{(0,t)}(R_E(\iota^{-1}\circ L_E\circ\Phi^c)_sy^{\ast}dm(s).
\end{eqnarray*}
We pass to the limit as $k\to\infty$ and find
\begin{eqnarray*}
(\iota(\hat{w}(t)))y^{\ast} & = & y^{\ast}(\phi(0))-(L_E\phi)y^{\ast}+\int_{(0,t)}(R_E(\hat{w}^c)_s)y^{\ast}dm(s)\\
& & +\int_{(0,t)}(R_E(\iota^{-1}\circ L_E\circ\Phi^{\ast})_s)y^{\ast}dm(s),
\end{eqnarray*}
or
\begin{eqnarray*}
(\iota(\hat{w}(t)))y^{\ast} & = & (\iota(\phi(0)))y^{\ast}-(L_E\phi)y^{\ast}+\left(\int_{(0,t)}R_E(\hat{w}^c)_sdm(s)\right)y^{\ast}\\
& & +\left(\int_{(0,t)}R_E(\iota^{-1}\circ L_E\circ\Phi^{\ast})_sdm(s)\right)y^{\ast}.
\end{eqnarray*}
So, $\hat{w}\in W_{\phi}$ satisfies Eq. (5.5)
 on $(0,t_0]$, and $\iota(\hat{w}(0))=\iota(\phi(0))-L_E\phi$. By uniqueness, $\hat{w}(t)=w(\phi)(t)$ for $0\le t\le t_0$.

10. Analogously we obtain that every subsequence of $(w_j)_1^{\infty}$ has a further subsequence which uniformly converges to $w(\phi)$ on $[0,t_0]$. This implies that the sequence 
 $(w_j)_1^{\infty}$ uniformly converges to $w(\phi)$ on $[0,t_0]$.
\end{proof}

\begin{prop}
(Local solution of Eq. (5.4) with initial data in $B_{nA}$) (i) Let $\phi\in B_{nA}$. 
Then there is a uniquely determined map
$v:[-h,t_0]\to\C^n$, $v=v^{\phi}$, with property (A) and $v_0=\phi$
which satisfies the integral equation (5.4),
\begin{equation*}
\iota(v(t))-L_Ev_t=\iota(\phi(0))-L_E\phi+\int_{(0,t)}R_E
v_sdm(s)
\end{equation*}
for all $t\in(0,t_0]$.

(ii) There is a constant $c_1\ge0$ with
$$
\sup_{-h\le t\le t_0}|v^{\phi}(t)|\le c_1|\phi|\,\,\text{for
all}\,\,\phi\in B_{nA}.
$$

(iii) For every bounded  sequence of maps $\phi_j\in B_{nA}$, $j\in\N$, which
converges pointwise to a map $\phi\in
B_{nA}$ we have
$$
v^{\phi_j}(t)\to
v^{\phi}(t)\,\,\text{as}\,\,j\to\infty,\,\,\text{for
each}\,\,t\in[0,t_0].
$$
\end{prop}

\begin{proof}
1. (Construction of a map $v$) Let $\phi\in B_{nA}$ be given. Set $w=w(\phi)$ and $w^c=w^{\phi^c}$. 
Notice that
$$
\iota(\phi(0))=\iota(\phi(0))-L_E\phi+L_E\phi=\iota(w(0))+L_E(\phi^c)_0
$$
and define $v:[-h,t_0]\to\C^n$ by $v_0=\phi$ and
$$
\iota(v(t))=\iota(w(t))+L_E(\phi^c)_t\,\,\text{for
all}\,\,t\in[0,t_0].
$$
Property (A) of $v$ follows by concatenation (Proposition 3.6 (ii))
from the facts that $\phi=v_0$ has property (A), that $w$ is
continuous (so it has property (A)), and that $\iota^{-1}\circ L_E\circ\Phi^c$ has
property (A) (see Proposition 5.4). It follows that all restrictions
$v|[-h,t]$, $0\le t\le t_0$, and all segments $v_t$, $0\le t\le
t_0$, have property (A); we have $v_t\in B_{nA}$ on $[0,t_0]$.
Proposition 3.5 (ii) yields that
$$
[0,t_0]\ni s\mapsto R_Ev_s\in(\C^n)^{\ast\ast}
$$
has property (A) (with respect to the weak-star topology on $(\C^n)^{\ast\ast}$), and the weak-star integrals
$$
\int_{(0,t)}R_Ev_sdm(s),\quad 0<t\le t_0,
$$
exist according to Proposition 4.2.

2. Proof of $L_E(\phi^c)_t=L_Ev_t$ for all $t\in[0,t_0]$ : Both
$(\phi^c)_t$ and $v_t$ belong to $B_{nA}$, and for $-h\le
u\le-\Delta_e$ we have $-h\le t+u\le t_0-\Delta_e\le0$, hence
$$
(\phi^c)_t(u)=\phi^c(t+u)=\phi(t+u)=v(t+u)=v_t(u).
$$
Apply Proposition 5.3.

3. We verify Eq. (5.4): Let $t\in(0,t_0]$. Then
\begin{eqnarray*}
\iota(v(t))-L_Ev_t & = & \iota(v(t))-L_E(\phi^c)_t=\iota(w(t))\\
& = &
\iota(\phi(0))-L_E\phi+\int_{(0,t)}R_E(w^c)_sdm(s)\\
& & +\int_{(0,t)}R_E(\iota^{-1}\circ L_E\circ\Phi^c)_sdm(s)\\
& = &
\iota(\phi(0))-L_E\phi+\int_{(0,t)}(R_E(w^c)_s+R_E(\iota^{-1}\circ L_E\circ\Phi^c)_s)dm(s)\\
& = &  \iota(\phi(0))-L_E\phi+\int_{(0,t)}R_E((w^c)_s+(\iota^{-1}\circ L_E\circ\Phi^c)_s)dm(s).
\end{eqnarray*}
For $s\in[0,t]$ and $u\in I$ with $s+u\le0$ the definition of
$w^c=w^{\phi^c}$ (see Proposition 5.5 (i)) shows
\begin{eqnarray*}
v_s(u) & = & v(s+u)=\phi(s+u)=\phi^c(s+u)=w^c(s+u)+(\iota^{-1}\circ L_E\circ\Phi^c)(s+u)\\
& = & (w^c)_s(u)+(\iota^{-1}\circ L_E\circ\Phi^c)_s(u)
\end{eqnarray*}
and in case $0<s+u$,
\begin{eqnarray*}
v_s(u) & = & v(s+u)=w(s+u)+\iota^{-1}L_E(\phi^c)_{s+u}\\
& = & w^c(s+u)+(\iota^{-1}\circ L_E\circ\Phi^c)(s+u)\\
& = & (w^c)_s(u)+(\iota^{-1}\circ L_E\circ\Phi^c)_s(u).
\end{eqnarray*}
Together,
$$
v_s=(w^c)_s+(\iota^{-1}\circ L_E\circ\Phi^c)_s,
$$
and the last integral is $\int_{(0,t)}R_Ev_sdm(s)$, and Eq. (5.4) holds.

4. (Uniqueness) Let $\phi\in B_{nA}$ and $v:[-h,t_0]\to\C^n$ with
$v_0=\phi$ and property (A) be given. Due to Proposition 3.5 (ii) also the map
$[0,t_0]\ni s\mapsto R_ev_s\in(\C^n)^{\ast\ast}$ has property (A), and for each $t\in(0,t_0]$ the weak-star integral
$\int_{(0,t)}R_Ev_sdm(s)\in(\C^n)^{\ast\ast}$ exists, see Proposition 4.2.
Suppose that $v$ satisfies Eq. (5.4) for $0<t\le t_0$. In the sequel we show that
$v$ is uniquely determined by these properties.

\smallskip

Notice first that due to property (A) $v$ is bounded. Define
$w:[0,t_0]\to\C^n$ by
$$
\iota(w(t))=\iota(v(t))-L_Ev_t.
$$
In particular,
$$
\iota(w(0))=\iota(\phi(0))-L_E\phi.
$$
Using Eq. (5.4) and the boundedness of $v$ we get for $0<u<t\le t_0$ the estimate
\begin{eqnarray*}
|w(t)-w(u)| & = & |\iota(w(t))-\iota(w(u))|=\left|\int_{(0,t)}R_Ev_sdm(s)- \int_{(0,u)}R_Ev_sdm(s)\right|\\
& = & \sup_{|y^{\ast}|\le1}\left|\int_{(0,t)}(R_Ev_s)y^{\ast}dm(s)- \int_{(0,u)}(R_Ev_s)y^{\ast}dm(s)\right|\\
& = &  \sup_{|y^{\ast}|\le1}\left|\int_{(u,t)}(R_Ev_s)y^{\ast}dm(s)\right|\\
& \le & \sup_{|y^{\ast}|\le1}(t-u)|R|_{L_c(C_{cn},\C^n)}\sup_{-h\le a\le t_0}|v(a)||y^{\ast}| 
\end{eqnarray*}
which shows that $w$ is continuous at points $t\in(0,t_0]$. Continuity at $t=0$ follows from the estimate
\begin{eqnarray*}
|w(t)-w(0)| & = & |\iota(w(t))-\iota(w(0))|=|\iota(v(t))-L_Ev_t-(\iota(\phi(0))-L_E\phi)|\\
& = &|\iota(v(t))-L_Ev_t-(\iota(v(0))-L_Ev_0)|=\left|\int_{(0,t)}R_Ev_sdm(s)\right|\\
& & \text{(see Eq. (5.4)}\\
& \le & (t-0)|R|_{L_c(C_{cn},\C^n)}\sup_{-h\le a\le t_0}|v(a)|
\quad\text{(as before)}
\end{eqnarray*}
for $0<t\le t_0$.

As in part 2 we get
$$
L_Ev_t=L_E(\phi^c)_t\,\,\text{for all}\,\,t\in[0,t_0].
$$
In order that $v$ be uniquely determined it remains to show that $w$
satisfies Eq. (5.5), which has a unique solution due to Proposition
5.5. The proof of this begins with a look at the map
$w^c=w^{\phi^c}$ from Proposition 5.5, i. e.,
$$
w^c(t)=w(t)\,\,\text{on}\,\,[0,t_0]\,\,\text{and}\,\,w^c(t)
=\phi^c(t)-(\iota^{-1}\circ L_E\circ\Phi^c)(t)\,\,\text{on}\,\,(-\infty,0].
$$
Let $s\in[0,t_0]$. For every $u\in I$ with $s+u<0$ we have
\begin{eqnarray*}
(w^c)_s(u) & = &
w^c(s+u)=\phi^c(s+u)-(\iota^{-1}\circ L_E\circ\Phi^c)(s+u)\\
& = & v(s+u)-(\iota^{-1}\circ L_e\circ \Phi^c)(s+u),
\end{eqnarray*}
and in case $0\le s+u$,
\begin{eqnarray*}
(w^c)_s(u) & = & w^c(s+u)=w(s+u)=v(s+u)-\iota^{-1}(L_Ev_{s+u})\\
& = & v(s+u)-\iota^{-1}(L_E(\phi^c)_{s+u})=v(s+u)-(\iota^{-1}\circ L_E\circ\Phi^c)(s+u).
\end{eqnarray*}
Hence
$$
(w^c)_s=v_s-(\iota^{-1}\circ L_e\circ\Phi^c)_s\,\,\text{on}\,\,[0,t_0].
$$
Now Eq. (5.4) yields
\begin{eqnarray*}
\iota(w(t))  &  =  & \iota(v(t))-L_Ev_t=\iota(\phi(0))-L_E\phi+\int_{(0,t)}R_Ev_sdm(s)\\
& = &
\iota(\phi(0))-L_E\phi+\int_{(0,t)}R_E((w^c)_s+(\iota^{-1}\circ L_E\circ\Phi^c)_s)dm(s)\\
& = &
\iota(\phi(0))-L_E\phi+\int_{(0,t)}R_E(w^c)_sdm(s)\\
& & +\int_{(0,t)}R_E(\iota^{-1}\circ L_E\circ\Phi^c)_sdm(s)
\end{eqnarray*}
for all $t\in(0,t_0]$, which is Eq. (5.5).

5. Proof of (ii). For $\phi\in B_{nA}$ let $v^{\phi}$ denote the
unique map $v:[-h,t_0]\to\C^n$ with property (A) which satisfies $v_0=\phi$ and Eq.
(5.4) for $0<t\le t_0$. By the construction of a solution in parts 1-3 and by uniqueness,
$$
\iota(v^{\phi}(t))=\iota(w(\phi)(t))+L_E(\phi^c)_t\,\,\text{for}\,\,0\le
t\le t_0.
$$
Using this and Proposition 5.5 (iii) we obtain
\begin{eqnarray*}
|v^{\phi}(t)| & = & |\iota(v^{\phi}(t))|\le|\iota(w(\phi)(t))|+|L|_{L_c(C_{cn},\C^n)}|(\phi^c)_t|\\
& \le & |w(\phi)(t)|+|L|_{L_c(C_{cn},\C^n)}|\phi|\le (c_0+|L|_{L_c(C_{cn},\C^n)})|\phi|
\end{eqnarray*}
for $0\le t\le t_0$.

6. Proof of (iii). Suppose the sequence of the maps $\phi_j\in B_{nA}$, $j\in\N$,
is bounded and converges pointwise to a map $\phi\in
B_{nA}$. Let $t\in[0,t_0]$. Then
$$
B_{nA}\ni((\phi_j)^c)_t\to(\phi^c)_t\in
B_{nA}\,\,\text{pointwise on}\,\,I\,\,\text{as}\,\,j\to\infty,
$$
and
$$
\sup_{j,s}|((\phi_j)^c)_t(s)|\le\sup_{j,s}|\phi_j(s)|<\infty.
$$
Proposition 3.5 (ii), applied to $J=I$ and $J_0=\{0\}$,  yields
$$
L_E((\phi_j)^c)_t\to
L_E(\phi^c)_t\,\,\text{as}\,\,j\to\infty,
$$
with respect to the weak-star topology on $(\C^n)^{\ast\ast}$, which is equal to the norm-topology.
Now use the equations
$$
v^{\phi_j}(t)=w(\phi_j)(t)+\iota^{-1}(L_E((\phi_j)^c)_t)\,\,\text{for}\,\,j\in\N,
$$
the analogous equation for $v^{\phi}$, and Proposition 5.6.
\end{proof}

We proceed to maximal solutions of Eq. (5.4).

\begin{prop}
(i) For every $\phi\in B_{nA}$ there is a uniquely determined map
$y^{\phi}:[-h,\infty)\to\C^n$, $y=y^{\phi}$, with $y_0=\phi$ so that
all restrictions of $y$ to compact intervals and all segments $y_s$,
$0\le s$, have property (A), all weak-star integrals $\int_{(0,t)}R_Ey_sdm(s)$, $t>0$, exist, and Eq. (5.4),
$$
\iota(y(t))-L_Ey_t=\iota(\phi(0))-L_e\phi+\int_{(0,t)}R_Ey_sdm(s)
$$
holds for all $t>0$.

(ii) For each $t\ge0$ there exists a constant $c(t)\ge0$ so that for
all $\phi\in B_{nA}$ we have
$$
\sup_{-h\le s\le t}|y^{\phi}(s)|\le c(t)|\phi|.
$$

(iii) For every bounded sequence of maps $\phi_j\in B_{nA}$, $j\in\N$,
which converges pointwise to a map
$\phi\in B_{nA}$ we have
$$
y^{\phi_j}(t)\to
y^{\phi}(t)\,\,\text{as}\,\,j\to\infty,\,\,\text{for
each}\,\,t\ge-h.
$$
\end{prop}

\begin{proof}
1. (A solution on $[-h,2t_0]$) Choose $t_0>0$ according to
Proposition 5.5. Let $\phi\in B_{nA}$ and consider $v^{\phi}$ as in
Proposition 5.7. Then $\chi=(v^{\phi})_{t_0}\in B_{nA}$, and
Proposition 5.7 yields a map $v^{\chi}$. For $t\in[-h,2t_0]$ define
\begin{eqnarray*}
Y(t) & = & v^{\phi}(t)\,\,\text{if}\,\,-h\le t\le t_0,\\
Y(t) & = & v^{\chi}(t-t_0)\,\,\text{if}\,\,t_0<t\le 2t_0.
\end{eqnarray*}
Using $v^{\chi}(0)=\chi(0)=v^{\phi}(t_0)$ and concatenation (Proposition 3.6
(ii)) we infer that $Y:[-h,2t_0]\to\C^n$ has property (A), which in
turn implies that all segments $Y_t$, $0\le t\le 2t_0$, have
property (A), and for all $t\in(0,2t_0]$ the weak-star integral $\int_{(0,t)}R_EY_sdm(s)$ exists, 
see Proposition 4.2. We have
$$
Y_t=v^{\chi}_{t-t_0}\quad\text{for}\quad t_0<t\le 2t_0
$$ 
since
$t_0<t\le 2t_0$ and $s\in I$ and $t_0<t+s$ yield
$$
Y_t(s)=Y(t+s)=v^{\chi}(t+s-t_0)=v^{\chi}_{t-t_0}(s)
$$
and $t_0<t\le 2t_0$ and $s\in I$ and $t+s\le t_0$ yield
\begin{eqnarray*}
Y_t(s) & = & Y(t+s)=v^{\phi}(t+s)=v^{\phi}(t-t_0+t_0+s)\\
& = & v^{\phi}(t_0+(t-t_0)+s)=\chi(t-t_0+s)=v^{\chi}(t-t_0+s)=v^{\chi}_{t-t_0}(s).
\end{eqnarray*}
For $t_0<t\le 2t_0$ we obtain
\begin{eqnarray*}
\iota(Y(t))-L_EY_t & = &
\iota(v^{\chi}(t-t_0))-L_E(v^{\chi})_{t-t_0}\\
& = & \iota(v^{\chi}(0))-L_E(v^{\chi})_0+\int_{(0,t-t_0)}R_E(v^{\chi})_sdm(s)\\
& = &
\iota(v^{\phi}(t_0))-L_E(v^{\phi})_{t_0}+\int_{(t_0,t)}R_E(v^{\chi})_{s-t_0}dm(s)\\
& = &
\iota(v^{\phi}(0))-L_E(v^{\phi})_0+\int_{(0,t_0)}R_E(v^{\phi})_sdm(s)\\
& & +\int_{(t_0,t)}R_E(v^{\chi})_{s-t_0}dm(s)\\
& = &
\iota(\phi(0))-L_E\phi+\int_{(0,t_0)}R_EY_sdm(s)+\int_{(t_0,t)}R_EY_sdm(s)\\
& = & \iota(\phi(0))-L_E\phi+\int_{(0,t)}R_EY_sdm(s).
\end{eqnarray*}
Also on $(0,t_0]$ we get
$$
\iota(Y(t))-L_EY_t=\iota(\phi(0))-L_E\phi+\int_{(0,t)}R_EY_sdm(s).
$$
It follows that the preceding equation holds for all
$t\in(0,2t_0]$.

2. (Uniqueness and the estimate from part (ii), both on $[-h,2t_0]$).
Suppose $Y:[-h,2\,t_0]\to\C^n$ has property (A), and $Y_0=\phi\in
B_{nA}$, and Eq. (5.4) holds on $(0,2\,t_0]$. Proposition 5.7 yields
$Y(t)=v^{\phi}(t)$ on $[0,t_0]$. The map $\hat{Y}:[-h,t_0]\ni
t\mapsto Y(t+t_0)\in\C^n$ (which has property (A)) satisfies Eq.
(5.4) on $(0,t_0]$ since for $0<t\le t_0$ we have
\begin{eqnarray*}
\iota(\hat{Y}(t))-L_E\hat{Y}_t & = &
\iota(Y(t+t_0))-L_EY_{t+t_0}=\iota(\phi(0))-L_E\phi\\
& & +\int_{(0,t+t_0)}R_EY_sdm(s)\\
& = & \iota(\phi(0))-L_E\phi+\int_{(0,t_0)}R_EY_sdm(s)+\int_{(t_0,t+t_0)}R_EY_sdm(s)\\
& = & \iota(v^{\phi}(t_0))-L_E(v^{\phi})_{t_0}+\int_{(0,t)}R_EY_{s+t_0}dm(s)\\
& = & \iota(\hat{Y}(0))-L_E\hat{Y}_0+\int_{(0,t)}R_E\hat{Y}_sdm(s).
\end{eqnarray*}
With $\chi=(v^{\phi})_{t_0}$ $(=Y_{t_0}=\hat{Y}_0\in B_{nA})$,
$$
\iota(\hat{Y}(t))-L_E\hat{Y}_t=\iota(\chi(0)
)-L_E\chi+\int_{(0,t)}R_E\hat{Y}_sdm(s)
$$
for $0<t\le t_0$. Proposition 5.7 yields
$\hat{Y}=v^{\chi}$. It follows that $Y$ is uniquely determined by
$\phi$. Moreover,
\begin{eqnarray*}
|Y(t)| & \le & c_1|\phi|\,\,\text{as}\,\,t\in[0,t_0]\,\,\text{and}\\
|Y(t)| & = & |Y((t-t_0)+t_0)|=|\hat{Y}(t-t_0)|\le c_1|\chi|\le
c_1c_1|\phi|\,\,\text{on}\,\,[t_0,2t_0],
\end{eqnarray*}
see Proposition 5.7 (ii). Hence
$$
|Y(t)|\le(1+c_1+c_1^2)|\phi|\,\,\text{for all}\,\,t\in[-h,2t_0].
$$

3. (Continuous dependence on $[-h,2t_0]$ with respect to pointwise
convergence) Suppose the sequence of the maps $\phi_j\in B_{nA}$, $j\in\N$, is 
bounded and converges pointwise to $\phi\in B_{nA}$. For $j
\in \N$ let $Y_j:[-h,2t_0]\to\C^n$ denote the unique map with property (A) and $Y_{j0}=\phi_j$ which satisfies
 Eq. (5.4) on $(0,2t_0]$. Similarly
let $Y:[-h,2t_0]\to\C^n$ denote the unique map  with property (A) and $Y_0=\phi$ which satisfies Eq. (5.4) on
$(0,2t_0]$. Using Proposition 5.7
(ii) and (iii) we infer that the restrictions of the maps $Y_j$ to
$[-h,t_0]$ are uniformly bounded and converge pointwise to $Y$ on
this interval, for $j\to\infty$. It follows that the sequence of the maps
$\chi_j=Y_{jt_0}\in B_{nA}$, $j\in\N$, is bounded and
converges pointwise to $\chi=Y_{t_0}$ as $j\to\infty$. Recall from
part 2 that the maps $\hat{Y}_j$, $j\in\N$, have property (A) and
solve Eq. (5.4) on $(0,t_0]$, and $\hat{Y}_{j0}=\chi_j$ for all
$j\in\N$. Now Proposition 5.7 yields that the maps $\hat{Y}_j$
converge pointwise to $\hat{Y}$ as $j\to\infty$. It follows that the maps $Y_j$
converge pointwise to $Y$ everywhere on
$[-h,2t_0]$, for $j\to\infty$.

4. Using induction and arguments as in parts 1-3 of the proof one
finds that for every $k\in\N$ and for each $\phi\in B_{nA}$ there is
a uniquely determined map $Y^{k,\phi}:[-h,kt_0]\to\C^n$ with
$(Y^{k,\phi})_0=\phi$ and property (A) so that Eq. (5.4) holds on
$(0,kt_0]$. Moreover statements analogous to (ii) and (iii) hold for
the maps $Y^{k,\phi}$, $\phi\in B_{nA}$ and $k\in\N$. For $\phi\in B_{nA}$ and $t\ge-h$ define
$Y^{\phi}(t)=Y^{k,\phi}(t)$ with $k\in\N$ such that $t\le kt_0$. Now it is easy to complete
the proof of Proposition 5.8.
\end{proof}

Using Eq. (5.4), Proposition 3.4 for $R_E$, and boundedness of
$y^{\phi}$ on bounded intervals we deduce the next result.

\begin{cor}
For each $\phi\in B_{nA}$ the map
$$
w:[0,\infty)\ni t\mapsto y^{\phi}(t)-\iota^{-1}(L_Ey^{\phi}_t)\in\C^n
$$
is continuous.
\end{cor}

\begin{proof}
(As in part 4 of the proof of Proposition 5.7) Set $y=y^{\phi}$. Let $t_1\ge0$. Choose $k\in\N$ with $t_1<kt_0$. Let $c=\sup_{-h\le t\le kt_0}|y(t)|<\infty$. For $0< u<t\le kt_0$ we have
\begin{eqnarray*}
|w(t)-w(u)| & = & |\iota(w(t))-\iota(w(u))| = \left|\int_{(0,t)}R_Ey_sdm(s)-\int_{(0,u)}R_Ey_sdm(s)\right|\\
& = & \sup_{|y^{\ast}|\le1}\left|\int_{(0,t)}(R_Ey_s)y^{\ast}dm(s)-\int_{(0,u)}(R_Ey_s)y^{\ast}dm(s)\right|\\
& = &  \sup_{|y^{\ast}|\le1}\left|\int_{(u,t)}(R_Ey_s)y^{\ast}dm(s)\right|\\
& \le &(t-u) \sup_{|y^{\ast}|\le1}\sup_{u<s<t}|(R_Ey_s)y^{\ast}|\le(t-u)|R|_{L_c(C_{cn},\C^n)}c
\end{eqnarray*}
(with Proposition 3.4 (ii)). This yields continuity of $w$ at points $t_1>0$. Continuity at $t=0$ follows from the estimate
\begin{eqnarray*}
|w(t)-w(0)| & = & |\iota(w(t))-\iota(w(0))| =|\iota(y^{\phi}(t))-L_Ey_t^{\phi}-(\iota(y^{\phi}(0))-L_Ey_0^{\phi})|\\ 
& = & \left|\int_{(0,t)}R_Ey_sdm(s)\right|\quad\text{(see Eq. (5.4)}\\
& \le & (t-0)|R|_{L_c(C_{cn},\C^n)}c\quad\text{(as before)} 
\end{eqnarray*}
for $0<t\le kt_0$.
\end{proof}

Observe that all maps $I\to\C^n$ which are constant on $[-h,0)$
belong to $B_{nA}$.

\begin{definition}
For $j\in\{1,\ldots,n\}$ let $X_j:\R\to\C^n$ denote the map which
is given by $X_j(t)=0$ for $t<0$ and
$$
X_j(t)=y^{\phi_j}(t)\,\,\text{for}\,\,t\ge0
$$
where $\phi_j(t)=0$ on $[-h,0)$ and $\phi_{jk}(0)=\delta_{jk}$ for
$k\in\{1,\ldots,n\}$. The matrix-valued map
$$
X:\R\to\C^{n\times n}
$$
with columns $X_j$ is called fundamental solution of Eq. (1.2).
\end{definition}

Of course, the discontinuous columns of the fundamental solution
are only solutions to the integral equation (5.4),  and not solutions of the differential equation (1.2), but we prefer
to use the standard terminology here. Observe in the sequel that
$X_{jk}(t)$ is the coefficient in the $j$-th column and $k$-th row
of the matrix $X(t)$.

For every $j\in\{1,\ldots,n\}$ all restrictions of $X_j$ to
compact intervals and all segments $X_{jt}$, $t\in\R$, have the
property (A), due to Propositions 5.8 and 3.6 (ii).

Recall that the solutions $[-h,\infty)\to\C^n$ of the initial
value problem (5.1)-(5.2) define a strongly continuous semigroup of
continuous linear operators $S(t):C_{cn}\to C_{cn}$,$t\ge0$, and
that there exist constants $c\ge0$ and $\gamma\in\R$ with
\begin{equation}
|S(t)|_{L_c(C_{cn},C_{cn})}\le c\,e^{\gamma t}\,\,\text{for all}\,\,t\ge0.
\end{equation}

\begin{cor}
(Growth of the fundamental solution) Suppose the estimate (5.6)
holds. Then for each
$j\in\{1,\ldots,n\}$ and for all $t\ge0$ we have
$$
|X_j(t)|\le c\,e^{\gamma t}.
$$
\end{cor}

\begin{proof}
Let $j\in\{1,\ldots,n\}$. Choose a sequence of continuous maps
$\phi_k\in C_{cn}$, $k\in\N$, with $|\phi_k|=1$ which converges
pointwise to $X_{j0}$. For $k\in\N$ let $v_k:[-h,\infty)\to\C^n$
denote the continuous solution of the IVP (5.1)-(5.2) with $v_{k0}=\phi_k$.
Proposition 5.2 shows that the continuous map $v_k$ also is a solution to Eq.
(5.4). All restrictions of $v_k$ to compact
intervals have property (A). It follows that $v_k=y^{\phi_k}$, for
all $k\in\N$. Let $t\ge0$. Proposition 5.8 (iii) yields
$$
(S(t)\phi_k)(0)=v_k(t)=y^{\phi_k}(t)\to
y^{X_{j0}}(t)=X_j(t)\,\,\text{as}\,\,k\to\infty.
$$
Using $|\phi_k|=1$ for all $k\in\N$ we infer
$$
|X_j(t)|=\lim_{k\to\infty}|(S(t)\phi_k)(0)|\le c\,e^{\gamma t}.
$$
\end{proof}
\section{Inhomogeneous equations}

Let a continuous map $q:[0,\infty)\to\C^n$ be given, in addition to
$L$ and $R$ from the previous section. A solution of the
inhomogeneous equation
\begin{equation}
\frac{d}{dt}(v-L\circ V)(t)=Rv_t+q(t)
\end{equation}
is a continuous map $v:[-h,\infty)\to\C^n$ so that
$$
[0,\infty)\ni t\mapsto v(t)-(L\circ V)(t)\in\C^n
$$
with $V:[0,\infty)\ni t\mapsto v_t\in C_{cn}$
is differentiable and satisfies Eq. (6.1) for all $t\ge0$. Solutions
are uniquely determined by their initial data $v_0\in C_{cn}$.

We look for a {\it particular solution} $p$ with $p_0=0$. Applying
Proposition 4.3 to the maps $X_j:\R\to\C^n$ and to the continuous
functions $q_j:[0,\infty)\to\C$, $j=1,\ldots,n$, we obtain that
the maps $p_j:\R\to\C^n$ given by
$$
p_j(t)=\int_{(0,t)}q_j(s)X_j(t-s)dm(s)\,\,\text{for}\,\,t>0
$$
and $p_j(t)=0$ for $t\le0$ are continuous. It follows that for
every $t>0$ the map
$$
(0,t)\ni s\mapsto X(t-s)q(s)\in\C^n
$$
is Lebesgue-integrable, and that the map $p=\sum_{j=1}^np_j$ is
continuous with
$$
p(t)=\int_{(0,t)}X(t-s)q(s)dm(s)\,\,\text{for}\,\,t>0
$$
and $p(t)=0$ for $t\le0$.

\begin{prop}
(i) For every $j\in\{1,\ldots,n\}$, $L_EX_{j0}=0$.

(ii) For every $t>0$ there is a bounded sequence of continuous maps from
$[0,t]^2$ into $\C^n$ which converges pointwise to the map
$$
[0,t]^2\ni(r,s)\mapsto\iota^{-1}( R_E(X_{s-r}q(r)))\in\C^n.
$$

(iii) For every $t>0$ the map
$$
(0,t)\ni s\mapsto R_E(X_{t-s}q(s))\in(\C^n)^{\ast\ast}
$$
has a weak-star integral, and
$$
\iota(Rp_t)=\int_{(0,t)}R_E(X_{t-s}q(s))dm(s).
$$
\end{prop}

\begin{proof}
1. Assertion (i) follows by means of Proposition 5.3.

2. Proof of (ii). Let $t>0$, $j\in\{1,\ldots,n\}$. As
$X_j|[-t-h,t]$ has property (A) we obtain from Proposition 3.5 (ii)
a sequence of continuous maps $\chi_k:[-t,t]\to(\C^n)^{\ast\ast}$,
$k\in\N$, which is uniformly bounded and converges pointwise to
the map
$$
[-t,t]\ni u\mapsto R_EX_{ju}\in(\C^n)^{\ast\ast}
$$
as $k\to\infty$, with respect to the weak-star topology on $(\C^n)^{\ast\ast}$, which is equal to the norm topology.
As $\iota:\C^n\to(\C^n)^{\ast\ast}$ is a norm-preserving isomorphism it follows that
the maps
$$
\iota^{-1}\circ\chi_k:[-t,t]\to\C^n,\,\,k\in\N,
$$
are continuous, uniformly bounded and converge pointwise to the map
$$
[-t,t]\ni u\mapsto \iota^{-1}(R_EX_{ju})\in\C^n
$$
as $k\to\infty$. Moreover, for $k\to\infty$ the continuous
maps
$$
[0,t]^2\ni(s,r)\mapsto q_j(r)(\iota^{-1}\circ \chi_k)(s-r)\in\C^n,\,\,k\in\N,
$$
converge pointwise to the map
$$
[0,t]^2\ni(r,s)\mapsto q_j(r)(\iota^{-1}(R_EX_{j,s-r})\in\C^n
$$
for $k\to\infty$. The sequence of these maps also is uniformly bounded.
Use that $\iota^{-1}$ and $R_E$ are linear, and sum up over $j\in\{1,\ldots,n\}$.

3. Proof of (iii). Let $t>0$. Corollary 4.6 gives that for
every $j\in\{1,\ldots,n\}$ the map
$$
(0,t)\ni s\mapsto R_E(q_j(s)X_{j,t-s})\in(\C^n)^{\ast\ast}
$$
has a weak-star integral, with
$$
\iota(Rp_{jt})=\int_{(0,t)}R_E(q_j(s)X_{j,t-s})dm(s).
$$
Summation over $j\in\{1,\ldots,n\}$ and linearity now yield the
assertion.
\end{proof}

\begin{prop}
The map $p$ is a solution to Eq. (6.1).
\end{prop}

\begin{proof}
1. It is sufficient to derive the equation
\begin{equation}
p(t)-Lp_t=\int_{(0,t)}Rp_sdm(s)+\int_{(0,t)}q(s)dm(s)
\end{equation}
for all $t>0$ as the integrands are continuous, the integrals are
Riemann integrals, and Eq. (6.1)
follows by differentiation.

2.  For $j\in\{1,\ldots,n\}$ set $L_j=pr_j\circ L\in(C_{cn})^{\ast}$. Recall the maps $ L_{jk}\in (C_c)^{\ast}$, for $j,k$ in $\{1,\ldots,n\}$, their extensions $ L_{jke}:B\to\C$  and the associated complex
Borel measures $\mu_{jk}$ from Section 3,
$$
d\mu_{jk}=h_{jk}d|\mu_{jk}|
$$
with  Borel
functions $h_{jk}:I\to\C$ with values on the unit circle. Recall also the definition of the extensions $L_{je}:B_n\to\C$ for $j\in\{1,\ldots,n\}$.
Analogously we have maps $R_j,R_{jk},R_{jke},R_{je}$, complex
Borel measures $\lambda_{jk}$, and measurable functions
$g_{jk}:I\to\C$ with values on the unit circle,
$$
d\lambda_{jk}=g_{jk}d|\lambda_{jk}|.
$$
Let $t>0$ and $j\in\{1,\ldots,n\}$ be given for the remainder of the proof. Then
\begin{eqnarray*}
pr_j(p(t)-Lp_t) & = &
p_j(t)-L_jp_t=p_j(t)-L_{je}p_t=p_j(t)-\sum_{k=1}^nL_{jke}p_{kt}\\
& = & p_j(t)-\sum_{k=1}^n\int p_{kt}d\mu_{jk}
\end{eqnarray*}
with
$$
\int p_{kt}d\mu_{jk}=\int
p_k(t+u)d\mu_{jk}(u)=\int p_k(t+u)h_{jk}(u)d|\mu_{jk}|(u)
$$
for each $k\in\{1,\ldots,n\}$. For such $k$ and for all $u\in I$ we have
$$
p_k(t+u)=\int_{(0,t)}\sum_{l=1}^nX_{lk}(t+u-r)q_l(r)dm(r)
$$
since in case $t+u\le 0$ both the left hand side and the integrand
are zero, and in case $0<t+u$ the definition of $p$ gives
\begin{eqnarray*}
p_k(t+u) & = & \int_{(0,t+u)}\sum_{l=1}^nX_{lk}(t+u-r)q_l(r)dm(r)\\
& = & \int_{(0,t)}\sum_{l=1}^nX_{lk}(t+u-r)q_l(r)dm(r)
\end{eqnarray*}
where the last equation holds because of $X_{lk}(t+u-r)=0$ for
$t+u<r<t$. Altogether we get
$$
\int
p_{kt}d\mu_{jk}=\int\left(\int_{(0,t)}\sum_{l=1}^nX_{lk}(t+u-r)q_l(r)dm(r)\right)h_{jk}(u)d|\mu_{jk}|(u)
$$
for every $k\in\{1,\ldots,n\}$.

3. We want to apply Corollary
2.4 (to Fubini's Theorem) and verify  measurability
properties first. Let $B_I$ and $B_{(0,t)}$ denote the Borel-$\sigma$-algebras on $I$ and on $(0,t)$, respectively. We show that the
maps
$$
I\times(0,t)\ni(u,r)\mapsto\sum_{l=1}^nX_{lk}(t+u-r)q_l(r)h_{jk}(u)\in\C,\quad
k=1,\ldots,n,
$$
are $B_I\times B_{(0,t)}$-measurable.

Each Borel function $q_l|(0,t)$ induces a function
$I\times(0,t)\ni(u,r)\mapsto q_l(r)\in\C$ for which preimages of
open sets have the form $I\times M$ with $M\in B_{(0,t)}$. Such
sets belong to the product $B_I\times B_{(0,t)}$. It follows that
the induced functions are $B_I\times B_{(0,t)}$-measurable.
Analogously all functions $h_{jk}:I\to\C$ induce functions
$I\times(0,t)\ni(u,r)\mapsto h_{jk}(u)\in\C$ which are $B_I\times B_{(0,t)}$-measurable. It
follows that the products
$$
I\times(0,t)\ni(u,r)\mapsto q_l(r)h_{jk}(u)\in\C
$$
are $B_I\times B_{(0,t)}$-measurable.

As the restrictions $X_{lk}|[-h,t]$ have property (A) it follows
that also the maps
$$
I\times(0,t)\ni(u,r)\mapsto X_{lk}(t+u-r)\in\C
$$
are pointwise limits of continuous functions. So they are
$B_I\times B_{(0,t)}$-measurable provided all continuous functions
$I\times(0,t)\to\C$ are $B_I\times B_{(0,t)}$-measurable. The last
statement is true since preimages of open sets under such
continuous maps are open in $I\times(0,t)$ and therefore countable
unions of sets $U\times V$ with $U$ and $V$ open subsets of $I$
and $(0,t)$, respectively; we have $U\times V\in B_I\times
B_{(0,t)}$.

Altogether we see that for all $l,k$ in $\{1,\ldots,n\}$ the
function
$$
I\times(0,t)\ni(u,r)\mapsto X_{lk}(t+u-r)q_l(r)h_{jk}(u)\in\C
$$
is $B_I\times B_{(0,t)}$-measurable.

4. It is now easy to see that for all $k,l$ in $\{1,\ldots,n\}$ and for every $r\in(0,t)$ the function
$$
I\ni u\mapsto X_{lk}(t+u-r)q_l(r)h_{jk}(u)\in\C
$$
is Lebesgue-integrable. Let $k\in\{1,\ldots,n\}$. We want to reverse the order of integration in the formula for
$\int p_{kt}d\mu_{jk}$. Notice first that because of property (A) the maps $X_{lk}$, $l\in\{1,\ldots,n\}$, are bounded on
$[-h,t]$. Using this in combination with the continuity of $q$ and
$|h_{jk}(t)|=1$ we infer that
$$
\int\left(\int_{(0,t)}\left|\sum_{l=1}^nX_{lk}(t+u-r)q_l(r)h_{jk}(u)\right|dm(r)\right)d|\mu_{jk}|(u)<\infty.
$$
So Corollary 2.4 applies, and the assertion of part (c) in Theorem 2.3 holds. We infer
\begin{eqnarray*}
\int p_{kt}d\mu_{jk} & = & \int\left(\int_{(0,t)}\sum_{l=1}^nX_{lk}(t+u-r)q_l(r)dm(r)\right)h_{jk}(u)d|\mu_{jk}|(u)\\
& = & \int\phi(u)d|\mu_{jk}|(u)\\
& & \text{with a Borel function}\,\,\phi:I\to\C\,\,\text{which satisfies}\\
& & \phi(u)=\int_{(0,t)}\sum_{l=1}^nX_{lk}(t+u-r)q_l(r)h_{jk}(u)dm(r)\\
& & \text{outside a Borel set}\,\,Z\subset I\,\,\text{with}\,\,|\mu_{jk}|(Z)=0\\
& = & \int_{(0,t)}\psi(r)dm(r)\\
& & \text{with a Borel function}\,\,\psi:(0,t)\to\C\,\,\text{which satisfies}\\
& & \psi(r)=\int\sum_{l=1}^nX_{lk}(t+u-r)q_l(r)h_{jk}(u)d|\mu_{jk}|(u)\\
& & \text{outside a Borel set}\,\,N\subset(0,t)\,\,\text{with}\,\,m(N)=0\\
& = & \int_{(0,t)}\left(\int\sum_{l=1}^nX_{lk}(t+u-r)q_l(r)h_{jk}(u)d|\mu_{jk}|(u)\right)dm(r)\\
& = & \int_{(0,t)}\sum_{l=1}^nq_l(r)\left(\int
X_{lk}(t+u-r)h_{jk}(u)d|\mu_{jk}|(u)\right)dm(r).
\end{eqnarray*}

5. Using the definition of $p$ and equations in parts 2 and 4 we get
\begin{eqnarray*}
pr_j(p(t) & - & Lp_t)=
\int_{(0,t)}\sum_{l=1}^nX_{lj}(t-r)q_l(r)dm(r)\\
& & -\sum_{k=1}^n\int_{(0,t)}\sum_{l=1}^nq_l(r)\left(\int X_{lk}(t+u-r)h_{jk}(u)d|\mu_{jk}|(u)\right)dm(r)\\
& = &
\int_{(0,t)}\left\{\sum_{l=1}^nq_l(r)\left[X_{lj}(t-r)-\sum_{k=1}^n\int X_{lk}(t+u-r)d\mu_{jk}(u)\right]\right\}dm(r)\\
& = &
\int_{(0,t)}\left\{\sum_{l=1}^nq_l(r)\left[X_{lj}(t-r)-\sum_{k=1}^n\int
X_{lk,t-r}d\mu_{jk}\right]\right\}dm(r).
\end{eqnarray*}
For every $l\in\{1,\ldots,n\}$  and $r\in(0,t)$ we have
\begin{eqnarray*}
(L_EX_{l,t-r})pr_j & = & (pr_j\circ L)_eX_{l,t-r}=L_{je}X_{l,t-r}\\
& = &
\sum_{k=1}^nL_{jke}[X_{l,t-r}]_k=
\sum_{k=1}^n\int X_{lk,t-r}d\mu_{jk}.
\end{eqnarray*}
Also,
\begin{eqnarray*}
X_{lj}(t-r)-(L_EX_{l,t-r})pr_j & = & pr_j(X_l(t-r))- (L_EX_{l,t-r})pr_j\\
& = & (\iota(X_l(t-r)))pr_j-(L_EX_{l,t-r})pr_j\\
& = & (\iota(X_l(0)))pr_j+\left\{\int_{(0,t-r)}R_EX_{l,s}dm(s)\right\}pr_j\\
& & \text{(with the integral equation satisfied}\\
& & \text{ by the fundamental solution,}\\
& &  \text{and with Proposition 6.1 (i))}\\
& = & pr_j(X_l(0))+\left\{\int_{(0,t-r)}R_EX_{l,s}dm(s)\right\}pr_j\\
& = & \delta_{jl}+\left\{\int_{(0,t-r)}R_EX_{l,s}dm(s)\right\}pr_j.
\end{eqnarray*}
Altogether we infer
\begin{eqnarray*}
pr_j(p(t) & - & Lp_t)=
\int_{(0,t)}\left\{\sum_{l=1}^nq_l(r)\left[X_{lj}(t-r)-\sum_{k=1}^n\int X_{lk,t-r}d\mu_{jk}\right]\right\}dm(r)\\
& = & \int_{(0,t)}\left\{\sum_{l=1}^nq_l(r)\left[X_{lj}(t-r)-(L_EX_{l,t-r})pr_j\right]\right\}dm(r)\\
& = & \int_{(0,t)}\left\{\sum_{l=1}^nq_l(r)\left[\delta_{lj}+\left(\int_{(0,t-r)}R_EX_{l,s}dm(s)\right)pr_j\right]\right\}dm(r)\\
& = & \int_{(0,t)}q_j(r)dm(r)+\int_{(0,t)}\sum_{l=1}^nq_l(r)\left(\int_{(0,t-r)}R_EX_{l,s}dm(s)\right)pr_jdm(r).
\end{eqnarray*}
(for $t>0$ and $j\in\{1,\ldots\}$).

6. For the proof of Eq. (6.2) it remains to show that we have
\begin{equation}
\int_{(0,t)}\left\{\sum_{l=1}^nq_l(r)\left[\int_{(0,t-r)}R_EX_{l,s}dm(s)\right]pr_j\right\}dm(r)
=\int_{(0,t)}pr_j(Rp_s)dm(s).
\end{equation}
Let $l\in\{1,\ldots,n\}$ and $r\in(0,t)$ be given. Using Proposition 4.1 we get
$$
\int_{(0,t-r)}R_EX_{l,s}dm(s)=\int_{(r,t)}R_EX_{l,s-r}dm(s)=\int_{(0,t)}R_EX_{l,s-r}dm(s)
$$
since $X_{l,s-r}=0$ for $s<r$. Consider the term on the left hand side in Eq. (6.3). We infer
\begin{eqnarray*}
\int_{(0,t)} & \,\, & \left\{\sum_{l=1}^n q_l(r)\left[\int_{(0,t-r)}R_EX_{l,s}dm(s)\right]pr_j\right\}dm(r)\\
& = &
\int_{(0,t)}\left\{\sum_{l=1}^nq_l(r)\left[\int_{(0,t)}R_EX_{l,s-r}dm(s)\right]pr_j\right\}dm(r)\\
& = &
\int_{(0,t)}\left\{\left[\int_{(0,t)}R_E\left(\sum_{l=1}^nq_l(r)X_{l,s-r}\right)dm(s)\right]pr_j\right\}dm(r)\quad\text{(with Proposition 4.1)}\\
& = &
\int_{(0,t)}\left\{\int_{(0,t)}R_E(X_{s-r}q(r))pr_jdm(s)\right\}dm(r).
\end{eqnarray*}
We want to apply Corollary 2.4 (Fubini's Theorem) and have to check the hypotheses.
Proposition 6.1 (ii) guarantees that there is a bounded sequence of continuous maps
$$
\chi_{\nu}:(0,t)^2\to\C^n,\,\,\nu\in\N,
$$
which converges pointwise to the map
$$
(0,t)^2\ni(r,s)\mapsto\iota^{-1}(R_E(X_{s-r}q(r)))\in\C^n
$$
as $\nu\to\infty$. The sequence of the continuous functions $pr_j\circ\chi_{\nu}:(0,t)^2\to\C$, $\nu\in\N$,
is bounded and converges pointwise to the function
$$
(0,t)^2\ni(r,s)\mapsto pr_j(\iota^{-1}(R_E(X_{s-r}q(r))))\in\C
$$
as $\nu\to\infty$. For $(r,s)\in(0,t)^2$ and $\nu\in\N$ we have
$$
 pr_j(\iota^{-1}(R_E(X_{s-r}q(r))))=(\iota(\iota^{-1}(R_E(X_{s-r}q(r)))))pr_j=(R_E(X_{s-r}q(r)))pr_j.
$$
It follows that the function
$$
(0,t)^2\ni(r,s)\mapsto R_E(X_{s-r}q(r))pr_j\in\C
$$
is $B_{(0,t)}\times B_{(0,t)}$-measurable and bounded (compare part 3 above), and we can apply Corollary 2.4. Hence
$$
\int_{(0,t)}\left\{\int_{(0,t)}(R_E(X_{s-r}q(r)))pr_jdm(s)\right\}dm(r)
$$
$$
=\int_{(0,t)}\left\{\int_{(0,t)}(R_E(X_{s-r}q(r)))pr_jdm(r)\right\}dm(s)
$$
$$
=\int_{(0,t)}\left\{\int_{(0,s)}(R_E(X_{s-r}q(r)))pr_jdm(r)\right\}dm(s)
\quad\text{(as $X_{s-r}q(r)=0\in B_{nA}$ for $s<r$)}
$$
$$
=\int_{(0,t)}\left\{\left[\int_{(0,s)}R_E(X_{s-r}q(r))dm(r)\right]pr_j\right\}dm(s)
$$
$$
=\int_{(0,t)}\{(\iota(Rp_s))pr_j\}dm(s)\quad\text{(with Proposition 6.1 (iii))}
$$
$$
=\int_{(0,t)}pr_j(Rp_s)dm(s),
$$
which completes the proof of Eq. (6.3).
\end{proof}

\begin{cor}
(Variation-of-constants formula) For every $\phi\in C_{cn}$ there
exists a unique solution $w:[-h,\infty)\to\C^n$ of Eq. (6.1) with
$w_0=\phi$. This solution $w=w^{\phi}$ is given by
\begin{eqnarray*}
w^{\phi} & = & v^{\phi}+p,\quad\text{or}\\
w^{\phi}(t) & = & v^{\phi}(t)+\int_{(0,t)}X(t-s)q(s)dm(s)\quad\text{for}\quad t>0,
\end{eqnarray*}
with the solutions $v^{\phi}$ to the IVP (5.1)-(5.2) from Proposition
5.1 and with the solution $p$ to Eq. (6.1) from Proposition 6.2.
\end{cor}

\begin{proof}
It remains to show that $w=w^{\phi}$  is a solution of Eq. (6.1) with $w_0=\phi$. We have $w_0=\phi+0=\phi$, $w$ is continuous, and for all $t\ge0$,
$$
w(t)-Lw_t=v(t)-Lv_t+p(t)-Lp_t.
$$
Define $W:[0,\infty)\to C_{cn}$ by $W(t)=w_t$.
It follows that the map $w-L\circ W:[0,\infty)\ni t\mapsto w(t)-(L\circ W)(t)\in\C^n$
is differentiable, with
$$
\frac{d}{dt}(w-L\circ W)(t)=Rv_t+Rp_t+q(t)=Rw_t+q(t)
$$
for all $t\ge0$, with the right derivative at $t=0$.
\end{proof}

\begin{cor}
Suppose the estimate (5.6) holds. Then for all $\phi\in C_{cn}$
and for all $t\ge0$ we have
$$
|w^{\phi}(t)|\le c\left(e^{\gamma
t}|\phi|+n\int_0^t\,e^{\gamma(t-s)}|q(s)|ds\right).
$$
\end{cor}

\begin{proof}
Use Corollaries 6.3 and 5.11 and
\begin{eqnarray*}
\left|\int_{(0,t)}X(t-s)q(s)dm(s)\right| & = & \left|\int_{(0,t)}\sum_{l=1}^nq_l(s)X_l(t-s)dm(s)\right|\\
\le\int_{(0,t)}n|q(s)|c\,e^{\gamma(t-s)}dm(s)
\end{eqnarray*}
for $t>0$.
\end{proof}





\begin{thebibliography}{999}

\bibitem{A} Amann, H., and J. Escher, {\it Analysis III.} Birkh\"auser, Basel, 2001.

\bibitem{D} Dieudonn\'e, J., {\it Foundations of Modern Analysis.}
Academic Press, New York, 1969.

\bibitem{DS} Dunford, N., and J. T. Schwartz, {\it Linear Operators.}
Vol. 1, Interscience, New York, 1957.

\bibitem{DvGVLW} Diekmann, O., van Gils, S. A., Verduyn Lunel, S. M., an
H. O. Walther, {\it Delay Equations: Functional-, Complex- and
Nonlinear Analysis}. Springer, New York, 1995.

\bibitem{H} Hale, J. K., {\it Theory of Functional Differential Equations.} Springer, New York, 1977.

\bibitem{HM} Hale, J. K., and K. R. Meyer, {\it A class of functional equations of neutral type}. Mem. Amer. Math. Soc. 76 (1967).

\bibitem{HVL} Hale, J. K., and S. M. Verduyn Lunel, {\it Introduction to
Functional Differential Equations}. Springer, New York, 1993.

\bibitem{HKWW} Hartung, F., Krisztin, T., Walther, H. O., and J. Wu,
{\it Functional differential equations with state-dependent delay:
Theory and applications}. In {\it HANDBOOK OF DIFFERENTIAL
EQUATIONS, Ordinary Differential Equations}, volume 3, pp. 435-545,
Canada, A., Drabek., P. and A. Fonda eds., Elsevier Science B. V.,
North Holland, Amsterdam 2006.

\bibitem{He} Henry, D., {\it Linear autonomous neutral functional
differential equations}. J. Differential Eqs. 15 (1974), 106-128.

\bibitem{NH} Nagel, R., and N. T. Huy, {\it Linear neutral partial functional differential equations: A semigroup approach}.
Int.J. of Mathematics and Mathematical Sciences 23 (2003), 1433-1445.

\bibitem{N} Natanson, I. P., {\it Theorie der Funktionen einer reellen Ver\"anderlichen.}
Akademie-Verlag, Berlin, 1961.

\bibitem{P} Poisson, S. D., {\it Sur les \'{e}quations aux
diff\'{e}rences mel\'{e}es}, Journal de l'Ecole Polytechnique, Paris, (1)
6, cahier 13 (1806), 126-147.

\bibitem{Ru} Rudin, W., {\it Real and Complex Analysis}.
McGraw-Hill, London et al., 1970.

\bibitem{W0} Walther, H. O., {\it \"{U}ber Ejektivit\"{a}t und periodische L\"{o}sungen bei
Funktionaldifferentialgleichungen mit verteilter Verz\"{o}gerung.}
Habilitation thesis, Munich, 1977.

\bibitem{W01} Walther, H. O., {\it On instability, $\omega$-limit sets, and periodic solutions of
nonlinear autonomous differential delay equations.} In {\it
Functional Differential Equations and Approximation of Fixed
Points}, H.O. Peitgen and H.O. Walther eds., 489-503, Lecture Notes
in Mathematics 730, Springer, Heidelberg 1979.

\bibitem{W8} Walther, H. O., {\it Semiflows for neutral equations with state-dependent
delays}.  Fields Institute Communications 64 (2013), 211-267.

\bibitem{W9} Walther, H. O., {\it Linearized stability for semiflows generated by a class of
neutral equations, with applications to state-dependent delays}. J. of  Dynamics  and  Differential Eqs. 22 (2010), 439-462.

\end{thebibliography}
\end{document}